\documentclass[12pt]{article}

\usepackage{amssymb,amsmath,amsthm,sectsty,url,mathrsfs,graphicx}
\usepackage[letterpaper,hmargin=1.0in,vmargin=1.0in]{geometry}
\usepackage[colorlinks,linkcolor=black,citecolor=black,filecolor=black,urlcolor=black]{hyperref}
\usepackage{color}

\sectionfont{\large}
\subsectionfont{\normalsize}
\numberwithin{equation}{section}

\newtheorem{atheorem}{Theorem}

\newtheorem{maintheorem}{Theorem}
\newtheorem{theorem}{Theorem}[section]
\newtheorem{lemma}[theorem]{Lemma}
\newtheorem{proposition}[theorem]{Proposition}
\newtheorem{corollary}[theorem]{Corollary}
\newtheorem{fact}[theorem]{Fact}
\newtheorem{definition}[theorem]{Definition}
\newtheorem{remark}[theorem]{Remark}
\newtheorem{conjecture}[theorem]{Conjecture}

\newtheorem{example}[theorem]{Example}

\newcommand{\Var}{{\bf Var}}

\newcommand{\PP}{\mathscr{P}}

\renewcommand{\Pr}{ \mathrm P}

\newcommand{ \rel}{ t_{\mathrm{rel}} }
\newcommand{ \reln}{ t_{\mathrm{rel}}^{(n)} }
\newcommand{ \mix}{ t_{\mathrm{mix}} }
\newcommand{ \mixn}{ t_{\mathrm{mix}}^{(n)} }
\newcommand{ \mixmu}{ t_{\mathrm{mix},\mu} }
\newcommand{ \mixmun}{ t_{\mathrm{mix},\mu_{n}}^{(n)} }

\newcommand{ \hit}{ \mathrm{hit} }

\newcommand{ \h}{ \mathrm{H} }
\newcommand{ \TV}{ \mathrm{TV} }

\newcommand{\eps}{\epsilon }
\renewcommand{\epsilon}{\varepsilon}
\newcommand{\del}{\delta }

\newcommand{\la}{\lambda}

\newcommand{\gO}{\Omega }

\newcommand{\ssfrac}[2]{\mbox{\footnotesize $\frac{#1}{#2}$}}
\newcommand{\half}{\ssfrac{1}{2}}

\newcommand{\pin}{\ensuremath{\pi_n}}

\DeclareMathSymbol{\leqslant}{\mathalpha}{AMSa}{"36} 
\DeclareMathSymbol{\geqslant}{\mathalpha}{AMSa}{"3E} 
\DeclareMathSymbol{\eset}{\mathalpha}{AMSb}{"3F}     
\renewcommand{\leq}{\;\leqslant\;}                   
\renewcommand{\geq}{\;\geqslant\;}                   

\newcommand{\N}{\mathbb N}
\newcommand{\R}{\mathbb R}
\newcommand{\Z}{\mathbb Z}

\usepackage[normalem]{ulem}

\begin{document}

\title{A technical report on hitting times, mixing and cutoff }
\author{Jonathan Hermon
\thanks{
University of Cambridge, Cambridge, UK. E-mail: {\tt jonathan.hermon@statslab.cam.ac.uk}. Financial support by
the  EPSRC grant EP/L018896/1.}
}
\date{}
\maketitle

\begin{abstract}
Consider a sequence of continuous-time irreducible reversible Markov chains and a sequence of initial distributions, $\mu_n$. Instead of performing a worst case analysis, one can study the rate of convergence to the stationary distribution starting from these initial distributions. The sequence is said to exhibit (total variation) $\mu_n$-cutoff if
the convergence to stationarity in total variation distance is abrupt, w.r.t.~this sequence of initial
distributions.

In this work we give a characterization of $\mu_n$-cutoff (and also of total-variation mixing) for an arbitrary sequence
of initial distributions $\mu_n$ (in the above setup). Our characterization is expressed  in terms of hitting times of sets which are ``worst"
(in some sense) w.r.t.~$\mu_n$. 

Consider a Markov chain on $\Omega$ whose stationary distribution is $\pi$.  Let $t_{\mathrm{H}}(\alpha) :=\max_{x
\in \Omega,A \subset \Omega :\,\pi(A)
\geq \alpha}\mathbb{E}_{x}[T_{A}]$ be the expected hitting time of the set of stationary probability at least $\alpha$ which is ``worst in expectation"
(starting from the worst starting state). The connection between $t_{\mathrm{H}}(\cdot) $ and the mixing time of the chain  was previously studied
by Aldous and later by Lov\'asz and Winkler, and was recently refined   by Peres and Sousi and independently by Oliveira. In this work we
further refine this connection and show that $\mu_n$-cutoff can be characterized in terms of concentration of hitting times (starting from $\mu_n$) of sets which are worst in expectation w.r.t.~$\mu_n$. Conversely, we construct  a counter-example
which demonstrates that in general cutoff (as opposed to cutoff w.r.t.~a certain sequence of initial distributions) cannot be characterized in this manner.

Finally, we also prove that there exists an absolute constant $C$ such that for any Markov chain  $\epsilon( t_{\h}(\epsilon)-t_{\h}(1-\epsilon))
\leq C\rel  |\log \epsilon|$, for all $0< \epsilon < 1/2$,  where $\rel$ is the inverse of the spectral gap of the additive symmetrization $\half(P+P^*)$.     
\end{abstract}

\newpage


\section{Introduction}
This work is a continuation of \cite{cutoff}, in which Starr's maximal inequality was used to characterize the cutoff phenomenon for reversible Markov chains in terms of concentration of hitting times of sets which are ``worst" in some sense. Here using the same technique we present several new related results.

The connection between mixing times and hitting times of sets which are ``worst" in some sense goes back to the pioneer work of Aldous \cite{aldous1982some} and the later body of work of Lov\'asz and Winkler \cite{lovasz1998mixing} on stopping rules. Recently, this connection was substantially refined by Peres and Sousi in \cite{peres2011mixing} and independently by Oliveira
\cite{oliveira2012mixing}. All of the aforementioned works considered sets whose hitting time is in some sense the ``worst in expectation"\footnote{In \cite{oliveira2012mixing} and \cite{peres2011mixing} the parameter $t_{\h}(\cdot)$ was considered and in \cite{aldous1982some,lovasz1998mixing} $\max_{\alpha \in (0,1)} \alpha t_{\h}(\alpha)$ was considered, where $t_{\h}(\cdot)$ is defined in Definition \ref{def: worstinexpectation}.}. The results in \cite{cutoff} give a more refined connection between hitting times and mixing times.

 We extend the results of \cite{cutoff} to the case cutoff (resp.~mixing) is considered only w.r.t.~a certain sequence of initial distributions (resp.~initial distribution). We show that in this setup one may interpret ``worst" above as ``maximizing the expected hitting time" w.r.t.~the considered sequence of initial distributions over all sets whose stationary probability is at least some $\alpha \in (0,1)$. Conversely, we show that this may fail when considering cutoff in the usual sense (not only
w.r.t.~a certain sequence of initial distributions).

\medskip

Generically, we shall denote the state space of a Markov chain by $\Omega
$ and its stationary distribution by $\pi$ (or $\Omega_n$ and $\pi_n$, respectively, for the $n$-th chain in a sequence of chains). We say that the chain is finite, whenever $\Omega$ is finite.
Let $(Y_t)_{t=0}^{\infty}$ be an irreducible Markov chain on a finite state
space $\Omega$ with transition matrix $P$ and stationary distribution $\pi$. We denote such a chain by $(\Omega,P,\pi)$. The \textit{time-reversal} of $P$ is $P^*$, given by $P^*(x,y):=\pi(y)P(y,x)/\pi(x)$.
A chain $(\Omega,P,\pi) $ is called \emph{\textbf{reversible}}
if $P=P^*$, i.e. $\pi(x)P(x,y)=\pi(y)P(y,x)$, for all $x,y \in \Omega$. The \textit{additive symmetrization} of $P$ is given by $Q:=\half (P+P^*)$.

\medskip

Periodicity issues can be avoided  by considering the continuous-time version of the chain, $(X_t)_{t
\geq 0}$.
This is a continuous time Markov chain whose heat kernel is defined by $H_t(x,y):=\sum_{k=o}^{\infty}\frac{e^{-t}t^k}{k!}P^t(x,y)$. It is a classic result of probability theory that for any initial condition
the distribution 
of $X_{t}$ converges to $\pi$ when $t$ goes to infinity. The object of the
theory of Mixing time for Markov chain is to
study the characteristic of this convergence (see \cite{aldous2000reversible,levin2009markov}
for self-contained
introductions to the subject). Throughout, we shall consider only continuous time chains, although all our results can be stated also in discrete time, assuming $P(x,x)\geq \delta$ for some $\delta>0$, for all $x \in \Omega$ (in fact, even if this fails, the results are still valid if one replaces the relaxation time by the absolute relaxation time; See \cite[Remark 1.8]{cutoff}).

\medskip

We denote by $\h_{\mu}^t$ ($\h_{\mu}$) the distribution of $X_t$ ($(X_t)_{t
\geq 0}$), given that the initial distribution is $\mu$.   When $\mu=\delta_x$ (where $\delta_x(y)=1_{x=y}$), for some $x \in \Omega$, we simply write $\h_x^t$ ($\h_x$).
We denote the set of probability distributions on a (finite) set $B$ by
$\mathscr{P}(B) $.  For any pair of distributions $\mu,\nu \in \mathscr{P}(B)$,  their \emph{\textbf{total-variation
distance}} is defined to
be
$$\|\mu-\nu\|_\mathrm{TV} := \frac{1}{2}\sum_{x } |\mu(x)-\nu(x)|=\max_{A \subset B}\mu(A)-\nu(A) =\sum_{x
\in B :\, \mu(x)>\nu(x)}\mu(x)-\nu(x).$$ 
The worst-case total variation distance at time $t$ is defined as $d(t):= \max_{x \in \Omega} \|\h_{x}^{t}- \pi\|_\mathrm{TV}$. For  $\mu \in \PP(\Omega)$ let $d_{\mu}(t) :=\|\h_{\mu}^{t}- \pi\|_\mathrm{TV}=\frac{1}{2} \sum_{y \in \Omega}|\sum_{x \in \Omega}\mu(x)H_{t}(x,y)-\pi(y)|$.
The $\epsilon$\textbf{-mixing-time} (resp.~w.r.t.~a fixed initial distribution $\mu$) is defined as  $$t_{\mathrm{mix}}(\epsilon) := \inf \left\{t : d(t) \leq
\epsilon \right\}, \quad (\text{resp. }t_{\mathrm{mix},\mu}(\epsilon)
:= \inf \left\{t : d_{\mu}(t) \leq
\epsilon \right\}). $$ 
When $\epsilon=1/4$ we simply write $\mix$ and $\mixmu $.

\medskip

Recall that if $(\Omega,P,\pi)$ is a finite  irreducible chain, then the additive symmetrization $Q$ is reversible and hence self-adjoint w.r.t.~the inner-product induced by $\pi$ (see Definition \ref{def: L_p distance of measures}). Thus $Q$ has $|\Omega|$ real eigenvalues. Throughout we shall denote them by $1=\lambda_1>\lambda_2 \geq \ldots \geq \lambda_{|\Omega|} \geq -1$ (where $\lambda_2<1$ by irreducibility).
Define the \textit{\textbf{spectral-gap}} and \emph{\textbf{relaxation-time}} of $P$ as $\lambda:=(1-\lambda_2)$ and $t_{\mathrm{rel}}:=1/\la$. The following general relation holds for reversible chains (see
$\cite{levin2009markov}$ Lemmas 20.5 and 20.11),
\begin{equation}
\label{eq: t_relintro}
\\  t_{\mathrm{rel}}|\log\left(2\epsilon \right)| \leq t_{\mathrm{mix}}(\epsilon)
\leq  t_{\mathrm{rel}}
|\log \left( \epsilon \min_x \pi(x) \right)|.
\end{equation}
Next, consider a sequence of such chains, $((\Omega_n,P_n,\pi_n): n \in \N)$,
each with its corresponding
worst-distance from stationarity $d_n(t)$, its mixing-time $t_{\mathrm{mix}}^{(n)}$,
etc..
Loosely speaking, the (total variation) \emph{\textbf{cutoff phenomenon}}
 occurs when over a negligible period of time, known as the \emph{\textbf{cutoff
window}}, the (worst-case) total variation distance (of a certain finite
Markov chain from its stationary distribution) drops abruptly from a value
close to 1 to near $0$. In other words, one should run the $n$-th chain until time $(1-o(1))\mixn $  for it to even slightly mix in total variation, whereas running it any further after time $(1+o(1))\mixn $  is essentially redundant.
Formally, we say that a sequence of chains exhibits a \emph{\textbf{cutoff}}
if the
following
sharp transition in its convergence to stationarity occurs:
$$\lim_{n \to \infty}t_{\mathrm{mix}}^{(n)}(\epsilon)/t_{\mathrm{mix}}^{(n)}(1-\epsilon)=1,
\text{ for every }0<\epsilon <1.$$
Similarly, for a sequence of initial distributions $\mu_n \in \PP(\Omega_n)$, we say that a sequence of chains exhibits a $\mu_n$-\emph{\textbf{cutoff}}
if
$$\lim_{n \to \infty}t_{\mathrm{mix},\mu_n}^{(n)}(\epsilon)/t_{\mathrm{mix},\mu_n}^{(n)}(1-\epsilon)=1,
\text{ for every }0<\epsilon <1.$$
We say that the sequence has a \emph{\textbf{cutoff window}} (resp.~$\mu_n$-cutoff window) $w_n$, if  $w_n=o(t_{\mathrm{mix}}^{(n)})$ (resp.~$w_n=o(t_{\mathrm{mix},\mu_{n}}^{(n)})$)
and for any $\epsilon \in (0,1)$ there exists $C_{\epsilon}>0$ such that
for all $n$ 
\begin{equation*}
\label{eq-cutoff-def}
 t_{\mathrm{mix}}^{(n)}(\epsilon)-t_{\mathrm{mix}}^{(n)}(1-\epsilon)  \leq C_{\epsilon}
w_n \quad \text{ (resp.} \quad t_{\mathrm{mix},\mu_n}^{(n)}(\epsilon)-t_{\mathrm{mix},\mu_n}^{(n)}(1-\epsilon)
 \leq C_{\epsilon}
w_n).   
\end{equation*}
We say that a family of chains satisfies the \emph{\textbf{product condition}}
 if $\lambda^{(n)}\mixn \to \infty$ as $n \to \infty$
(or equivalently, $t_{\mathrm{rel}}^{(n)}=o(\mixn)$).
The following well-known fact follows easily from the first inequality
in (\ref{eq: t_relintro}) (see e.g.~\cite{levin2009markov}, Proposition 18.4).
\begin{fact}
\label{fact: cutoffandtrel}
If a sequence of irreducible reversible chains exhibits a cutoff, then $t_{\mathrm{rel}}^{(n)}=o(t_{\mathrm{mix}}^{(n)})$.
\end{fact}
The following mixing parameter, introduced in \cite{cutoff}, shall play a key role in this work. 
\begin{definition}
\label{def: worstinprob}
Let $(\Omega,P,\pi)$ be an irreducible chain. Let $\mu \in 
\PP( \Omega)$, $\delta,\epsilon \in (0,1)$ and $t \geq 0$. We define $p_{\mu}(\delta,t):=\max_{B \subseteq \Omega:\, \pi(B) \geq \delta}\h_{\mu}[T_{B}
> t]$, where $T_B:=\inf\{t:X_t \in B \}$ is the \textbf{\textit{hitting time}} of
the set $B$. Set $p(\delta,t):= \max_{x \in \Omega }p_{x}(\delta,t)$. We define
\begin{equation*}
\begin{split}
\mathrm{hit}_{\delta,\mu}(\epsilon):=\min
\{t:p_{\mu}(\delta,t) \leq \epsilon \} \text{ and } \mathrm{hit}_{\delta}(\epsilon):=\min
\{t:p(\delta,t)\leq \epsilon \}.
\end{split}
\end{equation*}
\end{definition}
   
\begin{definition}
\label{def: worstinprobcutoff}
Let $(\Omega_n,P_n,\pi_n)
$ be a sequence of irreducible chains and let $\alpha \in (0,1)$. We say
that the sequence exhibits a $\mathrm{hit}_{\alpha}
$-cutoff (resp.~$\mathrm{hit}_{\alpha,\mu_n}
$-cutoff), if for every $0<\epsilon <1/4$, 
$$\mathrm{hit}_{\alpha}^{(n)}(\epsilon)-\mathrm{hit}_{\alpha}^{(n)}(1-\epsilon)
=o \left(\mathrm{hit}_{\alpha}^{(n)}(1/4)
\right)$$  $$(\text{respectively,} \quad \mathrm{hit}_{\alpha,\mu_n}^{(n)}(\epsilon)-\mathrm{hit}_{\alpha,\mu_n}^{(n)}(1-\epsilon)
=o \left(\mathrm{hit}_{\alpha,\mu_n}^{(n)}(1/4)
\right) ).
$$
\end{definition}

The main abstract result in \cite{cutoff} (Theorem 3) is the following theorem. 
\begin{atheorem}
\label{thm: psigmacutoffequiv}
Let $(\Omega_n,P_n,\pi_n)$ be a sequence of reversible irreducible
finite Markov chains. The following are equivalent:
\begin{itemize}
\item[(1)] The sequence exhibits a cutoff.  
\item[(2)] The
sequence exhibits a $\mathrm{hit}_{\alpha}$-cutoff for some $\alpha \in (0,1)$  and  $t_{\mathrm{rel}}^{(n)}=o(t_{\mathrm{mix}}^{(n)})
$.
\end{itemize}
\end{atheorem}
\begin{definition}
\label{def: worstinexpectation}
 Let $\mu \in \PP(\Omega)$ and $0 < \alpha < 1 $. Define
$$t_{\mathrm{H,\mu}}(\alpha):=\max_{A \subseteq \Omega :\,\pi(A)
\geq \alpha}\mathbb{E}_{\mu}[T_{A}] \quad \text{and} \quad t_{\mathrm{H}}(\alpha):=\max_x t_{\mathrm{H},x}(\alpha)=\max_{x
\in \Omega,A \subseteq \Omega :\,\pi(A)
\geq \alpha}\mathbb{E}_{x}[T_{A}].$$
\end{definition}
This work was greatly motivated by the results of Peres and Sousi in \cite{peres2011mixing}.
Similar results were obtained independently by Oliviera
\cite{oliveira2012mixing}. Both papers refined previous results of Aldous
\cite{aldous1982some} and of Lov\'asz and Winkler \cite{lovasz1998mixing}.
Their results share the general theme of describing
mixing-times in terms of hitting-times. Their approach relied on the theory
of random times to stationarity combined with a certain ``de-randomization" argument due to Aldous \cite{aldous1982some},
 which shows that for any reversible irreducible finite chain and
any stopping time $T$ (possibly w.r.t.~a larger filtration, allowing some external randomness) such that $X_{T} \sim \pi$, we have $\mix = O(\max_{x \in \Omega}
\mathbb{E}_x[T])$. As a  consequence, they showed that for any $0<\alpha<1/2$
(this was extended to $\alpha=1/2$ in \cite{griffiths2012tight}), there exist
some constants $c_{\alpha},c'_{\alpha}>0$ such that for every reversible irreducible
finite chain
\begin{equation*}
\label{eq: tHmixconnection}
c'_\alpha t_{\mathrm{H}}(\alpha) \leq \mix \leq c_\alpha t_{\mathrm{H}}(\alpha).
\end{equation*}
It is natural to ask whether the more studied mixing parameter $t_{\mathrm{H}}(\alpha)
$ could be used in Theorem \ref{thm: psigmacutoffequiv} instead of the mixing
parameter $\hit_{\alpha}(\cdot)$.

\medskip

The following theorem extends Theorem \ref{thm: psigmacutoffequiv} to arbitrary
starting distributions $\mu_n \in \PP(\Omega_n) $, such that $\rel^{(n)}=o(\mixmun)
$. In addition, it asserts that ``cutoff" w.r.t.~these initial distributions (i.e.~$\mu_n$-cutoff),
is in fact equivalent to concentration of hitting times of sets which are
``worst in expectation" w.r.t.~these initial distributions (in the sense
of Definition \ref{def: worstinexpectation}). 
\begin{maintheorem}
\label{thm: cutoffandsetsworstinexpectation}
Let $(\Omega_n,P_n,\pi_n)$ be a sequence of finite irreducible reversible
chains. Let $\mu_{n} \in
\PP(\Omega_{n})$ be such that $\reln=o(\mixmun(\delta))$, for some $0<\delta
< 1$. Then the following are equivalent:
\begin{itemize}
\item[i)] The sequence exhibits a $\mu_n$-cutoff.
\item[ii)] There exists some $\alpha \in (0,1)$ such that the sequence exhibits
a $\mathrm{hit}_{\alpha,\mu_{n}}$-cutoff.
\item[iii)] There exist some $\alpha \in (0,1)$  and a sequence of
sets $A_n \subseteq \Omega_n$ with $\pi_n(A_n) \geq \alpha$ satisfying that
$\mathbb{E}_{\mu_n}[T_{A_n}]=
t_{\h,\mu_n}^{(n)}(\alpha)$, such that $$\lim_{n \to \infty}\h_{\mu_{n}}[|T_{A_n}-\mixmun|<\epsilon
\mathbb{E}_{\mu_n}[T_{A_n}]]=1, \quad
\text{for every }\epsilon>0. $$
\end{itemize}
\end{maintheorem}
\begin{remark}
It was shown in \cite{cutoff} that the occurrence of  $\mathrm{hit}_{\alpha}$-cutoff  for some $\alpha \in (0,1/2]$ implies that the product condition holds. Example 8.2 in \cite{cutoff} demonstrates that this cannot be improved and that $\mathrm{hit}_{\alpha}$-cutoff  for some $\alpha \in (1/2,1)$ need not imply cutoff. A small variation of that example can be used to show that in general   $\mathrm{hit}_{\alpha,\mu_n}$-cutoff (even for small $\alpha$) need not imply $\mu_n$-cutoff. Thus the assumption that $\reln=o(\mixmun(\delta))$ for some $0<\delta
< 1$, in Theorem \ref{thm: cutoffandsetsworstinexpectation} cannot be removed.
\end{remark}
We say that a Markov chain is \emph{transitive} if for every $x,y \in \Omega$ there exists a bijection $\phi:\Omega \to \Omega$ such that (1) $\phi(x)=y$ and (2) for all $a,b \in \Omega$ we have that $P(a,b)=P(\phi(a),\phi(b))$.
\begin{corollary}
\label{cor: transitive}
Let $(\Omega_n,P_n,\pi_n)$ be a sequence of finite irreducible reversible
transitive
chains. Then the following are equivalent:
\begin{itemize}
\item[i)] The sequence exhibits a cutoff.

\item[ii)] The sequence satisfies the product condition, and for some sequence $x_n \in \Omega_n $ and some  $0<\alpha <1$, there exists a
sequence of
sets $A_n \subseteq \Omega_n$ with $\pi_n(A_n) \geq \alpha$ satisfying that
$\mathbb{E}_{x_n}[T_{A_n}]=
t_{\h,x_n}^{(n)}(\alpha)$, such that $$\lim_{n \to \infty}\h_{x_{n}}[|T_{A_n}-\mix^{(n)}|<\epsilon
\mathbb{E}_{x_n}[T_{A_n}]]=1,
\text{ for every }\epsilon>0.$$
\end{itemize}
\end{corollary}
The following proposition asserts that in general cutoff (as opposed to cutoff
w.r.t.~some sequence of initial distributions) cannot be characterized in
terms of the parameter $t_{\mathrm{H}}(\cdot) $.
\begin{proposition}
\label{rem: worstinexp}
There exists a sequence $(\Omega_n,P_n,\pi_n)$ of finite irreducible reversible
chains satisfying the product condition such that the following holds:
\begin{itemize}
\item
There exist $A_n
\subseteq \Omega_n$ and $x_n \in \Omega_n$  such that   $\pi_n(A_n) \geq 1/2
$ and $\mathbb{E}_{x_n}[T_{A_n}]=t_{\mathrm{H}}^{(n)}(1/2)
$.
\item
The distribution of the hitting times of $A_n$ are concentrated w.r.t.~the
initial states
$x_n$.
\item The sequence does not exhibit a $\hit_{1/2}$-cutoff.
\item
The sequence does not exhibit a cutoff.
\end{itemize} 
\end{proposition}
In Example
\ref{ex: aldous} we construct a sequence of chains which exhibits the behavior described
in Proposition \ref{rem: worstinexp}. 

\begin{remark}
It was shown in \cite{chen2013comparison} that a sequence of finite continuous-time Markov chains exhibits a cutoff iff $t_{L}^{(n)}(\epsilon)-t_{L}^{(n)}(1-\epsilon)=o(t_{L}^{(n)}(1/4)) $, where $t_{L}^{(n)}(\eps)$ is the $\eps$-mixing-time of the associated lazy chain. They also showed that the same holds for a sequence of fixed initial distributions. Hence in part (i) of Theorem \ref{thm: cutoffandsetsworstinexpectation} and of Corollary \ref{cor: transitive} we could have considered the lazy version of the chain, rather than its continuous-time version.
\end{remark}
Recall that the relaxation-time $\rel$ of $P$ is defined to be the inverse of the smallest non-zero eigenvalue of $I-\half(P+P^*)$. The main ingredient in the proof of Theorem \ref{thm: cutoffandsetsworstinexpectation} is the following proposition.\begin{proposition}
\label{prop: hitmix1/2intro}
Let $(\Omega,P,\pi)$ be a finite irreducible reversible Markov chain. Let
$\mu \in \PP(\Omega)$. Let $0 < \alpha, \eps < 1$ and let $0<\delta \leq \min \{ \eps,1-\epsilon,1-\alpha\} $. Then
\begin{equation}
\label{eq:quan1}
 {\rm hit}_{\alpha,\mu}(\epsilon+\delta)- \alpha^{-1} \log \left( \frac{ 1-\alpha}{\delta} \right)t_{\rm rel} \leq 
 \mixmu(\epsilon) \leq \mathrm{hit}_{\alpha,\mu}(\epsilon-\delta) +\frac{3}{2} \rel |\log
\delta|.
\end{equation}
The same holds when $\mu$ is omitted from the inequality. Moreover, even if the chain is non-reversible, we still have that
\begin{equation}
\label{eq:quan2}
 {\rm hit}_{\alpha,\mu}(\epsilon+\delta)-   \alpha^{-1} \log ( \frac{ 1-\alpha}{\delta})t_{\rm rel} \leq 
 \mixmu(\epsilon) \leq \mathrm{hit}_{1-\delta,\mu}(\epsilon-\delta) + \log
(2/\delta^{3}) \rel \log ( \rel \vee e).
\end{equation}  
\end{proposition}

\begin{conjecture}
\label{con:nonrev}
Let $t_*:=\inf\{t: \mathbf{Var}_{\pi}H_t f \leq e^{-2} \mathbf{Var}_{\pi}f \text{ for all }f \in \R^{\Omega} \}$, where for $g \in \R^{\Omega}$ we define $\mathbf{Var}_{\pi}g:=\|g-\mathbb{E}_\pi g\|_2^2 $, $\mathbb{E}_\pi g:= \sum_x \pi(x) g(x) $, $\|g\|_2^2:=\mathbb{E}_\pi g^2 $ and $H_tg(x):=\sum_y H_t(x,y)g(y)=\mathbb{E}_x g(X_t)$. 
 Then there exist some constants $c_{\eps},C_{\delta}>0$ such that for every irreducible finite Markov chain, for every $\eps \in (0,1) $ and every $\del \leq \min \{\eps,1-\eps\}$ we have that
\[ \max \{c_{\eps}t_*,\hit_{1-\del}(\eps + \delta) \} \leq \mix(\eps) \leq \hit_{1-\del}(\eps - \delta) +C_{\delta} t_*.  \]
\end{conjecture}
For more details regarding this conjecture and how one might prove it, see Conjecture \ref{con:nonrevmax} and Remark \ref{rem:weirdmaxinequality}. 

The advantage in the definition of $\rel$ as the relaxation-time of the additive symmetrization is that in this fashion it still admits an extremal characterization and thus in applications can be bounded from above via various standard comparison techniques. The disadvantage is that without reversibility, in general it no longer provides a lower bound on $\mix$ nor on the rate of exponential decay of variances. To see this consider the following example due to Chen \cite{chenphd} (similar examples can be found at \cite[Section 6]{tetali} and \cite[Example 9.26]{aldous2000reversible}). Let $\Omega =\{0,1\}^n$ and $P((x_{1},x_2,\ldots ,x_n),(x_2,x_3,\ldots,x_n,\eta))=\half$ for all $(x_{1},x_2,\ldots ,x_n) \in \Omega$ and $\eta \in \{0,1\}$. Observe that after $n$ steps the discrete-time chain attains (exactly) the uniform distribution on $\Omega$, while the relaxation-time of the additive-symmetrization is governed by that of simple random walk on the $n$-cycle and hence is of order $n^2$ (e.g.\ \cite[Section 12.3]{levin2009markov})  

We now comment about $t_*$ being a natural alternative extension of $\rel$ to the non-reversible setup, which is meant to eliminate the existence of such examples. 
\begin{remark}
It is possible to replace in all of our results $\rel$ by $t_*$ from the previous
conjecture. Let $H_t^*(x,y):=\frac{\pi(y)}{\pi(x)}H_t(x,y)$. Since $H_t^*$ is the dual operator of $H_t$ w.r.t.\ the inner-product on $\R^{\Omega} $ given by $\langle f,g \rangle_{\pi}:=\mathbb{E}_{\pi}[fg] $ we have that
\[t_{*}=\inf\{t: \langle H_t^* H_t f, f \rangle_{\pi} \leq  e^{-2}  \text{ for all }f \in \R^{\Omega} \text{ such that }\mathbb{E}_{\pi}f=0 \text{ and } \mathbb{E}_{\pi}f^2=1 \} \]
 Observe that under reversibility we have that $t_*=\rel$, while in general $t_* \leq \rel $. Moreover, observe that for all $k \in \N$ we have that  $\mathbf{Var}_{\pi}H_{kt_*} f \leq e^{-2k} \mathbf{Var}_{\pi}f$.

Since $H_t^* H_t $ is self-adjoint w.r.t.\ $\langle \bullet,\bullet \rangle_{\pi} $ it follows from the Courant-–Fischer–-Weyl min-max principle (and continuity w.r.t.\ $t$) that there exists some non-constant $f \in \R^{\Omega} $ such that $H_{t_{*}}^*H_{t_{*}}f=e^{-2}f $ and $\mathbb{E}_{\pi}f=0$.  It follows from general considerations (cf.\ \cite[the proof of Theorem 12.4]{levin2009markov}) that if $|f(x)|=\max_{y \in \Omega}|f(y)| $ then
\[\|H_{t_{*}}^*H_{t_{*}} (x,\bullet)-\pi(\bullet)  \|_{\TV} =\sum_y \half |H_{t_{*}}^*H_{t_{*}} (x,y) - \pi(y) | \geq \frac{1}{2e^{2}}. \]
It follows that $\|\h_{\mu}^{t_*}-\pi\|_{\TV} \geq \frac{1}{2e^2}$, where $\mu(z):=H_{t_{*}}^*(x,z)$, and so $\mix(\frac{1}{2e^{2}}) \geq t_*$. 
\end{remark}
\begin{remark}
Equation \eqref{eq:quan1} is essentially Proposition 1.7 from \cite{cutoff}, written (and proved) in a neater manner. Equation \eqref{eq:quan2} is new. It is not hard to extend the proof of Theorem \ref{thm: cutoffandsetsworstinexpectation} to the non-reversible setup if one replaces the assumption that $\reln=o(\mixmun(\delta)) $ with the assumption that $\reln \log(\reln \vee e ) =o(\mixmun(\delta)) $. Similarly, under the assumption that $\reln \log(\reln \vee e ) =o(\mix(\delta)) $, even without reversibility we have that cutoff is equivalent to the occurrence of $\hit_{\alpha}$-cutoff for some (and in fact, for all) $\alpha \in (0,1)$.
\end{remark}
In \cite{griffiths2012tight} the following general inequality was proved (without a reversibility assumption).

\begin{atheorem}
\label{thm: griffiths}
Fix $0<\epsilon< 1/2$. For every irreducible finite Markov chain
\[\epsilon t_{\mathrm{H}}(\epsilon) \leq t_{\mathrm{H}}(1/2).\]
\end{atheorem}
The following proposition offers an upper bound on $t_{\mathrm{H}}(\epsilon) - t_{\mathrm{H}}(1-\epsilon)$, which in some cases is considerably better than the bound in Theorem \ref{thm: griffiths} (in particular, this is the case under reversibility when the product condition holds). 
\begin{proposition}
\label{cor: hittingtimesinequality}
There exists an absolute constant $C>0$ such that for every finite irreducible chain  $(\Omega,P,\pi)$,
\begin{equation}
\label{eq: t_Heps}
t_{\mathrm{H,\mu}}(\epsilon) - t_{\mathrm{H,\mu}}(1-\epsilon) \leq C \eps^{-1}  \rel  , \text{ for every } 0<\eps <1/2 \text{ and } \mu \in \PP(\Omega).
 \end{equation}
 The same holds when $\mu$ is omitted from the l.h.s..
\end{proposition}
\begin{remark}
In the non-reversible setup it is possible that $t_{\mathrm{H}}(\epsilon ) \leq C \sqrt{\rel}$ for all $\epsilon \in (0,1)$. To see this, consider a walk on the $n$-cycle with a fixed clockwise bias. Here $\max_{x,y}\mathbb{E}_x[T_y] \leq Cn $, while the additive symmetrization is simple random walk on the $n$-cycle and hence  $\rel = \Theta(n^2)$ (e.g.\ \cite[Section 12.3]{levin2009markov}). Conversely, under reversibility we have that $t_{\mathrm{H}}(1/2) \geq \rel $ (e.g.~\cite[Lemma 4.39]{aldous2000reversible}) and thus \eqref{eq: t_Heps} implies that $t_{\mathrm{H}}(\epsilon) - t_{\mathrm{H}}(1-\epsilon) \leq C \eps^{-1}  t_{\mathrm{H}}(1/2)  $, recovering Theorem \ref{thm: griffiths}, up to a constant. 
\end{remark}  
%
%
\section{Starr's Maximal inequality}
\label{s:trelimplications}
In this section we present the machinery that will be utilized in the proof of the main results. The most important tool we shall utilize is Starr's
$L^2$ maximal inequality (Theorem
\ref{thm: maxergodic}). We start with a few basic definitions and facts. We denote $\Z_{+}:=\{n \in \Z: n \geq 0 \}$ and $\R_+:=\{t \in \R:t \geq 0 \}$.

\begin{definition}
\label{def: L_p distance of measures}
Let $(\Omega,P,\pi)$ be a finite chain. For $f \in \R^{\Omega}$, let $\mathbb{E}_{\pi}[f]:=\sum_{x \in \Omega}\pi(x)f(x)$ and $\Var_{\pi}f:=\mathbb{E}_{\pi}[(f-\mathbb{E}_{\pi}f)^{2}]$. The inner-product $\langle \cdot,\cdot \rangle_{\pi}$ and $L^{p} $ norms ($1 \leq p < \infty$) are
$$\langle f,g\rangle_{\pi}:=\mathbb{E}_{\pi}[fg] \text{ and } \|f \|_p:=\left( \mathbb{E}_{\pi}[|f|^{p}]\right)^{1/p}. $$
We identify $P^{k}$ and $H_t$ with the operators $P^{k},H_t:L^2(\R^{\Omega},\pi)
\to L^2(\R^{\Omega},\pi)$, defined by $P^{k} f(x):=\sum_{y \in \Omega}P^{k}(x,y)f(y) $  and
$H_t f(x):=\sum_{y \in \Omega}H_{t}(x,y)f(y) =\mathbb{E}_{x}[f(X_t)]$. 
\end{definition}
Recall that if $P$ is reversible, then also $H_t$ is self-adjoint w.r.t.~$\langle \cdot,\cdot \rangle_{\pi}$.
 Recall that the spectral gap can be characterized as $\la=\inf \{\frac{\langle (I-P) f,f\rangle_{\pi}}{\Var_{\pi}f}:f \in \R^\gO \text{ non-constant} \}$ and that $\frac{d}{dt}\Var_{\pi}H_tf= -\langle (I-P)H_{t} f,H_{t}f\rangle_{\pi} \leq -2 \la \Var_{\pi}H_tf $, from which the following  lemma,  known as the Poincar\'e inequality or the $L^2$-contraction Lemma, follows.
\begin{lemma}
\label{lem: L2exp}
Let $(\Omega,P,\pi)$ be a finite  irreducible  chain.  Let $f \in \R^{\Omega}$.  Then for all $t \geq 0$
\begin{equation}
\label{eq: L2contraction0}
\Var_{\pi}H_tf \leq e^{-2 \la t}\Var_{\pi}f.
\end{equation}
\end{lemma}
We now state a particular case of Starr's maximal inequality (\cite{starr1966operator} Proposition 3). The proof in the discrete time setup could be found in \cite{cutoff}. Let $f \in \R^{\Omega}$. Define
its \textbf{\textit{maximal function}}
by $f^{*}(x):=\sup_{t \geq 0}|H_{t}f (x)|$.
\begin{theorem}[Maximal inequality \cite{starr1966operator}]
\label{thm: maxergodic}
Let $(\Omega,P,\pi)$ be a finite reversible irreducible Markov chain. Then for every $f \in \R^{\Omega}$
\begin{equation}
\label{eq: ergodic1}
 \|f^{*} \|_{2} \leq 2 \|f \|_2.
\end{equation}
\end{theorem}
For any $B \subseteq \Omega$ and $s \in \R_+$, set $\rho(B):=\sqrt{\pi(B)(1-\pi(B))}=\sqrt{\Var_{\pi}1_B}
$ and $\sigma_{s}:=\rho(B)e^{- \la s} $.  Note that by Lemma \ref{lem: L2exp}, $\sigma_{s} \geq \sqrt{\Var_{\pi}H_{s}1_B}$. We define the \emph{good set for }$B$ \emph{from time}
$s$ \emph{within }$m$ \emph{standard-deviations} to be
\begin{equation}
\label{eq: goodsetforB}
G_{s}(B,m):=\left\{ y: |\h_y^t(B)-\pi(B)| < m\sigma_{s} \text{,
for all }t \geq s \right\}.
\end{equation}
The following corollary follows by combining Lemma \ref{lem: L2exp} with Theorem \ref{thm: maxergodic}. 
\begin{corollary}
\label{cor: maxergcor}
Let $(\Omega,P,\pi)$ be a finite reversible irreducible chain. Then 
\begin{equation}
\label{eq: G_t}
 \pi (G_{s}(B,m)) \geq 1-\frac{4}{m^{2}}, \text{ for all }B \subseteq \Omega,\, s\geq 0 \text{ and }m>0.
\end{equation}
\end{corollary}

\emph{Proof:}
For any $s \geq0$, let $f_s(x):=H_{s}(1_{B}(x)-\pi(B)) =\h_x^s(B)-\pi(B)$.  Then by Lemma \ref{lem: L2exp} and Theorem \ref{thm: maxergodic}
\begin{equation}
\label{eq: Good1}
\|f_s^*\|_2 \leq 2\|f_s\|_2 \leq 2e^{- \la s}\|1_{B}(x)-\pi(B) \|_{2}=2\sigma_{s}. 
\end{equation}
For any $t \geq 0$, $H_tf_s(x)=f_{t+s}(x)=\h_x^{t+s}(B)-\pi(B)$. 
Then, in the notation of Theorem \ref{thm: maxergodic}, $$f_s^*(x):=\sup_{t
\geq 0}|H_{t}f_s(x) |=\sup_{t \geq s}|\h_x^{t}(B)-\pi(B) |. $$
Hence $D:=\left\{x \in \Omega:f_s^{*}
(x) \geq m\sigma_{s}\right\}$ is the complements of $G_{s}(B,m)$. Thus by Markov inequality and (\ref{eq: Good1}) 
\begin{equation*}
1-\pi(G_{s}(B,m))=\pi(D) =\pi \left\{\left(f_s^{*}
\right)^2 \geq  (m\sigma_{s})^{2}\right\}\leq 4m^{-2}. \qed
\end{equation*}
\begin{conjecture}
\label{con:nonrevmax}
Let $(\Omega,P,\pi)$ be a finite irreducible Markov chain (not necessarily reversible). Define $t_*:=\inf\{t: \mathbf{Var}_{\pi}H_t f \leq e^{-2} \mathbf{Var}_{\pi}f \text{ for all }f \in \R^{\Omega} \}$ as in Conjecture \ref{con:nonrev}. Let $f^{*}_s(x):= \sup_{t \geq s}|H_{t}f (x)|$. Then for every $\eps \in (0,1) $  there exists a constant $C_{\eps}$ (independent of $(\Omega,P,\pi)$) such that for every $f \in \R^{\Omega}$ of mean zero (i.e.\ $\mathbb{E}_{\pi}f=0 $) we have that 
\begin{equation}
\label{eq: ergodic1nonrev}
 \mathbb{E}_{\pi}f^{*}_{C_{\eps}t_*}  =\| f^{*}_{C_{\eps}t_*} \|_1 \leq \eps \|f \|_\infty.
\end{equation}
\end{conjecture}
\begin{remark}
\label{rem:weirdmaxinequality}
If \eqref{eq: ergodic1nonrev} holds then one could prove the upper bound on $\mix(\eps) $ from Conjecture \ref{con:nonrev} by imitating the proof of the upper bound on $\mix(\eps)$ from Proposition \ref{prop: hitmix1/2intro}, with \eqref{eq: ergodic1nonrev} replacing \eqref{eq: ergodic1}. 
\end{remark}
\section{Inequalities relating $t_{\rm mix}(\cdot)$ and ${\rm hit}_{\cdot}(\cdot)$}
\label{sec: 3}
Our aim in this section is to obtain inequalities relating $t_{\rm mix}(\epsilon)$ and ${\rm hit_{\beta}(\delta)}$ for suitable values of $\beta$, $\epsilon$ and $\delta$ using Corollary \ref{cor: maxergcor}. As was shown in \cite{cutoff}, these two notions of mixing are intimately connected to each other. In this section we  refine the analysis from \cite{cutoff}. Corollary \ref{cor: Dc} below contains the more difficult half of Proposition \ref{prop: hitmix1/2intro}. 
\begin{lemma}
\label{lem: hitprob}
Let $(\Omega,P,\pi)$ be a finite irreducible reversible chain.
Let $\mu \in \PP(\Omega)$,  $\alpha,p \in (0,1)$,  $s\geq 0$ and  $B
\subseteq \Omega$. Then
\begin{equation}
\label{eq: hitmixct}
\pi(B)- \h_{\mu}^{\mathrm{hit}_{\alpha ,\mu}(p)+s}(B) \leq p\pi(B)+2(1-p)e^{- \la s}  \sqrt{(1-\alpha)^{-1} \pi(B)\pi(B^{c})}.
\end{equation}
\end{lemma}
\begin{proof}
Denote $\tau:=\mathrm{hit}_{\alpha ,\mu}(p)$,  $\rho(B):=\pi(B)\pi(B^{c}) $ and $\ell:=2e^{- \la s}\sqrt{(1-\alpha)^{-1}} $. Consider
$$H:=\left\{ y:
|\h_y^{t}(B)-\pi(B)| < \ell  \sqrt{
\rho(B)}
\text{
for all }t \geq s \right\}. $$
By Corollary \ref{cor: maxergcor}, 
$\pi(H) \geq \alpha$.
By the Markov property and the
definition of $H$, $$\h_{\mu}[X_{\tau+s}\in B \mid T_{H}
\leq \tau] \geq  \pi(B)-\ell  \sqrt{
\rho(B)}.$$ Since $\tau=\mathrm{hit}_{\alpha,\mu}(p)$ and  $\pi(H) \geq \alpha $, we get that $\h_{\mu}[ T_{H}\leq \tau] \geq 1-p$. Thus 
\begin{equation*}
\begin{split}
& \pi(B)- \h_{\mu}^{\tau+s}(B) \leq \pi (B)- \h_{\mu}[X_{\tau+s}\in B,T_{H}
\leq \tau] \\ & = \h_{\mu}[T_{H}
> \tau] \pi(B)+\h_{\mu}[T_{H}
\leq \tau] \left(\pi(B)-\h_{\mu}[X_{\tau+s}\in B \mid T_{H}
\leq \tau] \right) \\ & \leq p \pi(B)+(1-p)\left(\pi(B)-\h_{\mu}[X_{\tau+s} \in B \mid T_{H}
\leq \tau] \right) \leq p \pi(B)  +(1-p)\ell  \sqrt{\rho(B)} .
\end{split}
\end{equation*}
This concludes the proof of (\ref{eq: hitmixct}).
\end{proof}

\begin{corollary}
\label{cor: Dc}
Let $(\Omega,P,\pi)$ be a reversible irreducible finite chain. Let
$\mu \in \PP(\Omega)$. Let $0 < \eps,\alpha < 1$ and let $\delta \in (0,\eps) $. Denote $s:= 0 \vee \rel \log
\left(\frac{1-\epsilon+\delta}{ \sqrt{(1-\alpha) \epsilon \delta}} \right) $. Then
\begin{equation}
\label{eq:quan3}  
 \mixmu(\epsilon) \leq \mathrm{hit}_{\alpha,\mu}(\epsilon-\delta) +s .
\end{equation}
\end{corollary}
\begin{proof}
 Denote  $t:= \mathrm{hit}_{\alpha,\mu}(\epsilon-\delta)$. If $1-\epsilon+\delta < \sqrt{(1-\alpha)\epsilon \delta} $, then we can prove \eqref{eq:quan3} for $\delta'<\delta$ such that  $1-\epsilon+\delta' = \sqrt{(1-\alpha)\epsilon \delta'} $, so we may assume that $s= \rel \log
\left(\frac{1-\epsilon+\delta}{ \sqrt{(1-\alpha)\epsilon \delta}} \right) \geq 0 $.  Let $B \subseteq \gO $. Denote $\pi(B)=z$.
 By (\ref{eq: hitmixct}) and the fact that  $h(z):=(\epsilon-\delta)z+2\sqrt{\epsilon \delta z(1-z)} $ attains its maximum in $[0,1]$ at $z_{*}=\frac{\eps}{\eps+\del} $ and $h(z_{*})=\eps$, we get that $$ \pi(B)-\h_{\mu}^{t+s}(B) \leq (\epsilon-\delta)z+2\sqrt{\epsilon \delta z(1-z)} \leq \eps,$$
The claim now follows by maximizing over $B$. \end{proof}

\section{Proofs of Theorem \ref{thm: cutoffandsetsworstinexpectation} and Propositions \ref{prop: hitmix1/2intro} and \ref{cor: hittingtimesinequality}}
Let $A \subsetneq\ \Omega $.  Let  $Q_A$ (resp.~$P_A$) be the restriction of $Q$ (resp.~$P$) to $A$. Note that $Q_A$ and $P_A$ are substochastic.  The spectral gap of $P_A$, denoted by $\la(A)$, is defined as the minimal eigenvalue of $I-Q_A$ (it is often referred to as the Dirichlet eigenvalue of $I-Q$ on $A$). Denote $\Lambda(c):=\min_{A:\pi(A) \leq c }\la(A)$.
Let $\pi_{A}$ denote $\pi$ conditioned
on $A$ (i.e.~$\pi_{A}(y)=1_{y \in A}\pi(y)/\pi(A)$).
\begin{lemma}
\label{lem: AF1}
Let $(\Omega,P,\pi)$ be a finite irreducible  Markov chain. Let
$A \subsetneq\ \Omega$ be non-empty. Let $\alpha>0$ and $w \geq 0$. Let $B(A,t,\alpha):=\left\{y:\h_{y}
\left[ T_{A^c} >  t 
\right] \geq \alpha  \right\}$. Then
\begin{equation}
\label{eq: CM0}
\la(A) \geq \pi(A^c) \la \quad \text{and so} \quad \forall c \in (0,1),\, \Lambda(c) \geq (1-c)\la.
\end{equation}
\begin{equation}
\label{eq: CM1}
\frac{1}{\pi(A)} \h_{\pi}[T_{A^{c}} > t]=   \h_{\pi_{A}}[T_{A^{c}} > t]\leq e^{-\la(A)t} \leq e^{-\Lambda(\pi(A))t}  \leq   e^{-\la
\pi(A^{c}) t}, \quad \text{for all }t \geq 0.
\end{equation}
\begin{equation}
\label{eq: CM2}
\pi \left(B(A,t,\alpha)\right) \leq \pi(A)   e^{-2 \la(A)t}\alpha^{-2} \leq \pi(A)   e^{-2 \pi(A^c) \la t}\alpha^{-2}   .
\end{equation}
\end{lemma}
\begin{remark}
The inequality \eqref{eq: CM0} is well-known (cf.\ \cite[Eq.\ (1.4)]{spectral}). We include its proof for the sake of completeness.
\end{remark}
\emph{Proof of Lemma \ref{lem: AF1}:}
 Consider the inner-product $\langle \cdot,\cdot \rangle_{\pi_{A}}$ on $\R^A$ given by  $\langle f,g \rangle_{\pi_{A}}:=\sum_{a \in A}\pi_{A}(a)f(a)g(a)$. We identify every $g \in \R^{A} $ with its extension to $\Omega$ obtained by setting $g \equiv 0 $ on $\Omega \setminus A $. Using the fact that for all $g \in \R^A$ we have $\langle (I-Q) g,g\rangle _{\pi}=\langle (I-P) g,g\rangle _{\pi}$, together with the Perron-Frobenius (for the non-negativity) and  the Courant-Fischer variational characterization of eigenvalues we have 
\begin{equation}
\label{e:varchar}
\la(A)=\inf \left\{\langle (I-P_{A}) g,g\rangle _{\pi_{A}}/\langle g,g \rangle _{\pi_{A}}: g \in \R_+^A ,\, g \text{ non-constant}  \right\}.
\end{equation}
We may identify each $g \in \R^A$ with $g \in \R^{\Omega}$ s.t.~$g \equiv 0$ on $A^c$ and vice-versa.
Hence \[\la(A)=\inf \left\{\langle (I-P) g,g\rangle _{\pi}/\langle g,g \rangle _{\pi}:g \in \R_+^{\Omega} \text{ non-constant and }g=0~\text{on}~A^{c}  \right\}.\] Also observe that by the Cauchy-Schwarz inequality, for all $g\geq 0$ such that $g=0$ on $A^{c}$,  
\[\Var_{\pi}g=\mathbb{E} _{\pi} g^2-(\mathbb{E} _{\pi} g)^2 =\mathbb{E} _{\pi} g^2-(\pi(A) \mathbb{E} _{\pi_A} g)^2\geq \mathbb{E} _{\pi} g^2-[\pi(A)]^{2}\mathbb{E} _{\pi_A} g^2= \pi(A^{c}) \langle g,g \rangle _{\pi}.\]

Thus by the extremal characterization of the spectral gap
\begin{equation}
\label{eq:gap(A)gap}
\la=\inf \left\{\langle (I-P) g,g\rangle _{\pi}/\Var_{\pi}g:g \in \R^{\Omega} \text{ non-constant} \right\} \leq \la(A)/\pi(A^c). 
\end{equation}

Denote the heat kernel of the chain killed upon escaping $A$ by $H_t^A:=e^{-t} \sum_{k=0}^{\infty}(kP_A)^t/k! $. Let $g  \in \R^A$. Denote $g_t:=H_{t}^A g $ and $\|g\|_{A,p}^p:=\sum_{a \in A} \pi_A(a)|g(a)|^p $, for $1 \leq p < \infty$. Then by \eqref{e:varchar}
\begin{equation}
\label{eq:ddt}
\frac{d}{dt}\|g_{t}\|_{A,2}^2=\frac{d}{dt}\langle H_{t}^A g,H_{t}^Ag\rangle _{\pi_{A}}=-2\langle (I-P_{A})H_{t}^A g,H_{t}^Ag\rangle _{\pi_{A}} \leq -2 \la(A)\|g_{t}\|_{A,2}^2.
\end{equation}
Hence $\|g_{t}\|_{A,1}^2 \leq \|g_{t}\|_{A,2}^2 \leq \|g \|_{A,2}^2 e^{-2 \la(A)t} $. Taking $g=1_A$, by \eqref{eq:gap(A)gap} we get that
\begin{equation}
\label{eq:hitla(A)}
 \h_{\pi_A}[T_{A^c}>t]= \|g_{t}\|_{A,1} \leq \|g_{t}\|_{A,2} \leq \|1_{A}\|_{A,2} e^{- \la(A)t}=e^{- \la(A)t} \leq e^{-  \pi(A^c) \la t}.
\end{equation}

Write $B=B(A,t,\alpha)$. Then $B \subseteq \{a \in A: g_t^2(a)>\alpha^2 \} $. By \eqref{eq:hitla(A)}$$\pi(B)/\pi(A) = \pi_{A} (B) \leq \alpha^{-2} \|g_{t}\|_{A,2}^2 \leq \alpha^{-2}e^{-2  \pi(A^c) \la t}  . \qed $$
\begin{corollary}
\label{cor: easydirection}
Let $(\Omega,P,\pi)$ be a finite irreducible  Markov chain. Let $\mu \in \PP(\Omega)$,  $0 < \eps,\alpha < 1 $ and $0<\delta \leq 1- \eps $. Denote $s:=\frac{\log (\frac{1-\alpha}{\delta})}{\Lambda(1-\alpha)} \leq \frac{\log (\frac{1-\alpha}{\delta})}{\alpha \la } $.
Then$${\rm hit}_{\alpha,\mu}(\epsilon+ \delta )\leq t_{\rm mix,\mu}(\epsilon)+s \vee 0 .$$
\end{corollary}
\begin{proof}
By decreasing $\del$ if necessary, we may assume that $s \geq 0$.
Take an arbitrary set $A$ with $\pi(A)\geq \alpha$
and $\mu \in \PP ( \Omega)$. It follows by coupling of the chain with initial distribution $\h_{\mu}^t$
with the
stationary chain that for all $t \geq 0$
\begin{equation}
\label{eq: thitinequality3}
\h_{\mu}[T_A > t+s]\leq  d_{\mu}(t)+\h_{\pi}[T_A>s]\leq
d_{\mu}(t)+(1-\alpha)e^{-\la(A^c)s }\leq d_{\mu}(t)+\delta.
\end{equation}
where the penultimate  inequality follows from (\ref{eq: CM1}) and
the last from the choice of $s$.  Plugging $t=t_{\rm mix , \mu }(\epsilon)$ 
in (\ref{eq: thitinequality3})  and maximizing over $A$ with
$\pi(A)\geq \alpha $ concludes the proof. 
\end{proof}
\emph{Proof of Proposition \ref{prop: hitmix1/2intro}:} By combining Corollaries \ref{cor: Dc} and \ref{cor: easydirection} it remains only to show that $\mixmu(\epsilon) \leq \mathrm{hit}_{1-\delta,\mu}(\epsilon-\delta) + \log
(2/\delta^{3}) \rel \log ( \rel \vee e) $, for $0<\delta \leq \eps <1$. Set $t:=\log
(2/\delta^{3}) \rel \log ( \rel \vee e)$. Let $A \subseteq \gO$. For all $x \in \gO$ and $s \geq 0$ we have that $|\frac{d}{ds}H_s(x,A)|=|\delta_x H_s (P-I)1_A | \leq 1 $ and so by \eqref{eq: L2contraction0} (used to deduce that $\Var_\pi H_s 1_A \leq e^{-2 \la s}\Var_{\pi}1_A \leq \frac{e^{-2 \la s}}{4} $)
\[ \sum_x \pi(x) | \frac{d}{ds}[H_s(x,A)- \pi(A)]^{2} |  \leq 2 \sum_x \pi(x) |H_s(x,A)- \pi(A)| \leq 2 \sqrt{\Var_\pi H_s 1_A} \leq  e^{- \la s}.   \]
 Denote $g(x):=\int_{t}^{\infty} | \frac{d}{ds}[H_s(x,A)- \pi(A)]^{2} |ds $. Then $\|g\|_1 \leq  \int_{t}^{\infty} e^{- \la s}ds= e^{-\la t}/\la \leq \del^3 /2$ and  $ \Var_{\pi}H_t 1_A \leq \frac{1}{4}e^{-2 \la t} \leq \del^{6} /8 $. Let $B:=\{b \in \gO: g(b)+|H_t(b,A)- \pi(A) |^2> \delta ^{2}\} $. Then
\[\pi(B) \leq \delta^{-2} (\|g \|_1+\Var_{\pi}H_t 1_A) \leq \delta.  \]
From the definition of $B$ it follows that if $x \notin B$ then $\sup_{s \geq t} |H_s(x,A)- \pi(A)| \leq \delta $. Thus \[|\h_{\mu}^{\mathrm{hit}_{1-\delta,\mu}(\epsilon-\delta)+t}(A)-\pi(A)| \leq \h_{\mu}[T_{B^{c}}>\mathrm{hit}_{1-\delta,\mu}(\epsilon-\delta)]+\delta \leq \eps. \qed \]
\begin{corollary}
\label{prop: hitpqinequalities}
Let $(\Omega,P,\pi)$ be an irreducible finite Markov chain. Let $\mu \in \PP (\Omega)$. Let  $0< \epsilon < \delta  < 1$ and  $0<\beta \leq \gamma < 1$. Denote $s:=\frac{\log (\frac{1-\beta }{(1-\gamma)\epsilon^{2}})}{2\Lambda(1-\beta)} \leq \frac{1}{2\beta} t_{\mathrm{rel}} \log \left( \frac{1-\beta }{(1-\gamma)\epsilon^{2}} \right) $. Then
\begin{equation}
\label{eq: thitinequality1}
 \mathrm{hit}_{\gamma,\mu }(\delta) \leq \mathrm{hit}_{\beta,\mu }(\delta)
\leq  \mathrm{hit}_{\gamma,\mu }(\delta- \epsilon) +s.
\end{equation}
\end{corollary}
\begin{proof}
The first inequality in (\ref{eq: thitinequality1}) is trivial. We now prove the second inequality in (\ref{eq: thitinequality1}).  Let $A$ be an arbitrary set with $\pi(A) \geq \beta$. Using the notation from Lemma \ref{lem: AF1}, let $B:= B(A^{c},s,\epsilon):=\{y:\h_y[T_{A}>s]>\eps \} $. Then by (\ref{eq: CM2}) $$ \pi(B) \leq \epsilon^{-2}\pi(A^c)e^{-\log \left(
\frac{1-\beta}{(1-\gamma)\epsilon^{2}} \right)} \leq 1- \gamma.$$ Set $t_1:= \mathrm{hit}_{\gamma,\mu }(\delta- \epsilon)$. Then $\h_{\mu}[T_{B^c}>t_{1}] \leq \delta- \eps $. By the definition of $B$ and the Markov property,
\begin{equation*}
\h_{\mu}[T_{A}>t_{1}+s
\mid T_{B^{c}} \leq t_{1}] \leq \max_{d \notin B }\h_{d}[T_{A}>s
] \leq \eps.
\end{equation*}
Whence
\begin{equation*}
 \h_{\mu}[T_{A}>t_{1}+s] \leq \h_{\mu}[T_{B^{c}}>t_{1}]+\h_{\mu}[T_{A}>t_{1}+s \mid T_{B^{c}} \leq t_{1}] \leq (\delta-\epsilon)+\epsilon=\delta.
\end{equation*}
Since $A$ was arbitrary, this concludes the proof of (\ref{eq: thitinequality1}).
\end{proof}

\begin{proposition}
\label{prop: equivoftalphascutoff}
Let $(\Omega_n,P_n,\pi_n)$ be a sequence of finite irreducible reversible
chains. Let $\mu_n \in \PP(\Omega_n)$. Assume that $\rel^{(n)}=o(\mixmun(\delta) )$ for some $0 < \delta <1$. Then for all $\beta \in (0,1)$,
\begin{equation}
\label{eq: hitThetamix}
 (1-o(1))\mixmun(\delta)  \leq  \mathrm{hit}_{\beta,\mu_{n}}^{(n)}(\delta/2)
\leq  (1+o(1))t_{\mathrm{mix},\mu_n}^{(n)}(\delta/8).
\end{equation}
\begin{equation}
\label{eq: hitThetamix2}
 \rel^{(n)}=o(\mathrm{hit}_{\beta,\mu_{n}}^{(n)}(\delta/2) )
\end{equation}
   Moreover, (1) below implies (2):
\begin{itemize}
\item[(1)] There exists some $\alpha \in (0,1)$ such that
the sequence exhibits a $\mathrm{hit}_{\alpha,\mu_{n}}$-cutoff.
\item[(2)] For every $\alpha
\in (0,1)$ the sequence exhibits a $\mathrm{hit}_{\alpha,\mu_{n}}$-cutoff.
\end{itemize}
Moreover, if (1) holds for some $\alpha$ then 
\begin{equation}
\label{eq: ratiohit}
\lim_{n \to \infty}
\mathrm{hit}_{\beta,\mu_{n}}^{(n)}(1/4)/\mathrm{hit}_{\alpha,\mu_{n}}^{(n)}(1/4)
= 1, \quad \text{for all} \quad \beta \in (0,1).
\end{equation}
\end{proposition}
\begin{proof}
 First note that \eqref{eq: hitThetamix} follows from Proposition \ref{prop: hitmix1/2intro} together with the assumption that $\rel^{(n)}=o(\mixmun(\delta) )$, while \eqref{eq: hitThetamix2} follows from the \eqref{eq: hitThetamix}. Assume (1) holds for some $\alpha$.  This in conjunction with Corollary \ref{prop: hitpqinequalities} and \eqref{eq: hitThetamix2}  implies (\ref{eq: ratiohit}), which combined with another application of Corollary \ref{prop: hitpqinequalities} yield (2) (cf.~\cite[Proposition 3.6]{cutoff}). 
\end{proof}

 We now present two
lemmas regarding expected hitting times inequalities. Proposition \ref{cor:
hittingtimesinequality} follows by combining
these two lemmas. The first of which is simpler and gives better bounds for
some poruses. The second one gives better asymptotic in the sense that it
follows from it that $t_{\h}(\epsilon)-t_{\h}(1-\epsilon)
\leq C \rel \epsilon^{-1}$, for some absolute constant $C$, whereas the first
lemma only implies that $t_{\h}(\epsilon)-t_{\h}(1-\epsilon)
\leq C \rel  \epsilon^{-1} \log (1/\epsilon)$.
\begin{lemma}
\label{lem: frombigtobigger}
Let $(\Omega,P,\pi)$ be a finite  irreducible Markov chain. Let
$ \epsilon \in (0,1)$ and $k \geq 0$. Let $A \subseteq \Omega $ be such $\pi(A^{c})
\leq  \epsilon$. There exists $I=I(A,k) \subseteq \Omega$ with $\pi(I^{c})
\leq \frac{\epsilon}{2
\cdot 3^{k}}$, such that $$ \mathbb{E}_{z}[T_{A}]
\leq (9/4+k/2) (\log 3)/\la(A^c) \leq (9/4+k/2) (1-\epsilon)^{-1}
\rel \log 3, \text{ for all } z \in I.  $$
\end{lemma}
\begin{proof}
Fix $k \geq 0$. Denote $a:=(\log 3)/ \la(A^c) \leq \Lambda^{-1}(\epsilon) \log 3 \leq  (1-\epsilon)^{-1}
\rel \log 3 $. Consider $$I_{i}=I_{i}(A,k):=\left\{y:\h_{y}
\left[2T_{A} > (k+3i)a
\right] \leq 3^{-i}  \right\}.$$ Then by (\ref{eq: CM2}) we have that $$\pi(I_i^{c})
\leq (3^{-i})^{-2}\pi(A^{c}) e^{-(k+3i)\log 3} \leq
\epsilon(1/3)^{i+k}.$$ Let $I:= \bigcap_{i \in \N}I_i$. Then $\pi(I^{c}) \leq \sum_{i \in \N} \pi(I_i^c) \leq  \eps \sum_{i \in \N}3^{-(i+k)}=\frac{\epsilon}{2 \cdot 3^{k}}$.
By construction, $$\h_{z}
\left[ T_{A} -\frac{ka}{2}> \frac{3ia}{2} \right] \leq 3^{-i}, \text{ for all } z \in I \text{ and } i \in
\N.$$ Hence $\mathbb{E}_{z}[T_{A}] \leq \frac{3a/2}{1-\frac{1}{3}}  +ka= (9/4+k/2)a$ for all $z \in I$, as desired. 
\end{proof}
\begin{lemma}
\label{lem: fromhugetotinyineexpectation}
Let $(\Omega,P,\pi)$ be a finite irreducible  chain. Let $A \subseteq \Omega$ be non-empty.
Denote $\epsilon:=\pi(A)$. Let $t \geq 1$. There exists $J=J(A,t) \subseteq \Omega $ such that
\begin{itemize}
\item[(i)]
$\pi(J) \geq 1-3e^{-2t}/2$.
\item[(ii)]
For any $z \in J$ we have that $\mathbb{E}_{z}[T_{A}] \leq \rel [t+\log 2 (1+ 9\eps^{-1})+|\log \epsilon|]
$.
\item[(iii)]
For any $z \in J$ and $i\in \N$ we have that $\h_{z}[\la T_{A} \leq
 t+i\log 2 (1+\frac{3}{2 \eps})+|\log \epsilon|] \geq (1-2^{-i})^{2}.$
\end{itemize}
\end{lemma}
\begin{proof}
Denote $r=r(\epsilon,t):=
(t+|\log \epsilon| ) \rel$, $\ell:=\rel  \log
2$ and $s=s(\epsilon):=(3/2)\epsilon^{-1}\rel\log 2$.
For any $i \in \N$,  let $C_i:=\left\{y:\h_{y}[T_{A} \leq is] \geq  1- 2^{-i}
\right\}$. Then by  (\ref{eq: CM2}) we have that
\begin{equation}
\label{eq: Ciislarge}
\pi(C_i^{c}) \leq 
2^{2i}
\pi (A^c)e^{-3i \log 2 } = (1-\epsilon)2^{-i}.
\end{equation}
 Define $$J_{i}:=\left\{y:
\h_{y}^{r+i \ell}[C_{i}] \geq 1-2^{-i}=1- (1-\epsilon)2^{-i}-\epsilon 2^{-i}  \right\}.$$  Denote  $g_i:=H_{r+i \ell}1_{C_i}$.
By stationarity $\mathbb{E}_{\pi}[g_i]=\mathbb{E}_{\pi}[1_{C_i}]=\pi(C_i)
 \geq 1- (1-\epsilon)2^{-i}  $. Hence 
\begin{equation}
\label{eq: Bisusbetof}
J_i^{c} \subseteq \{x: g_i(x)\geq \mathbb{E}_{\pi}[g_i]+ \epsilon 2^{-i} \}.
\end{equation}
By
(\ref{eq: L2contraction0}) and the fact that $z_1(1-z_1)<z_2(1-z_2)$ if $z_1,z_2 \in [0,1]$ and $|z_1-1/2|>|z_2-1/2|$, $$\Var_{\pi} g_i \leq e^{-2 \la(r+i
\ell)}\Var_{\pi}1_{C_i} \leq \epsilon^2 2^{-2i}e^{-2t}\pi(C_i)\pi(C_i^{c})
 \leq \epsilon^2 2^{-3i}(1-2^{-i}) e^{-2t}.$$ By the Chebyshev's inequality and (\ref{eq: Bisusbetof})
$$\pi(J_i^{c}) 
\leq \epsilon^{-2}2^{2i}\Var_{\pi}g_{i} \leq 2^{-i}e^{-2t}(1-2^{-i}).$$
Denote $J:=\bigcap_{i=1}^{\infty}J_i$. By a union bound $$\pi(J^{c}) \leq
 \sum_{i \in \N} \pi(J_i^c) \leq  e^{-2t}(1+3/16+\sum_{i \geq 3}2^{-i})<3e^{-2t}/2. $$
Fix some $z \in J$ and $i \geq 1$. Note that because $ J \subseteq J_i $ we get from
the definitions of $J_i$ and $C_i$ together with the Markov property that
$$\h_{z} \left[T_{A} \leq r+i(\ell+s) \right] \geq \h_{z}[X_{r+i\ell} \in
C_i ] \min_{y \in C_i}\h_{y}[T_{A} \leq is] \geq (1-2^{-i})^{2} .$$
From this it is easy to verify that indeed $\mathbb{E}_z \left[T_A \right]
\leq r+6(s+\ell).$
\end{proof} 
The following lemma asserts that, for a fixed starting distribution $\mu$
such that $\rel$ is much smaller than $\mixmu $, a set $A$ which
is ``worst" in expectation (i.e.~$\mathbb{E}_{\mu}[T_{A}]=t_{\h,\mu}(\pi(A))$) is almost the ``\emph{worst in probability}"
(in the sense of Definition \ref{def: worstinprob}) for all times. By this
we mean that this is the case up to a small size and time  shifts and up
to a small difference in the chance of not being hit by any given time.
\begin{lemma}
\label{lem: worstinexp}
Let $(\Omega,P,\pi)$ be a finite  irreducible reversible chain.  Let $\mu
\in \PP(\Omega)$ and $0<\epsilon <1$. Let $A \subseteq \Omega $ be such that $\mathbb{E}_{\mu}[T_{A}]=t_{\h,\mu}(1-\epsilon)$
and $\pi(A) \geq 1-\epsilon$. Denote $\rho=\rho(\epsilon):= 3 (1-\epsilon)^{-1}
\rel \log 3$. Then for all $r \geq 1$
\begin{equation}
\label{eq: AB1}
\max_{B \subseteq \Omega: \pi(B) \geq 1-\epsilon/2} \h_{\mu}[T_{B}-T_{A} \geq r \rho] \leq r^{-1}. 
\end{equation}
In particular, for all $t \geq 0$, $r \geq 1$,  $q \in (0,1) $ and $B \subseteq \gO $ with $\pi(B^c) \leq \eps/2 $ we have that
\begin{equation}
\label{eq: AB2}
\begin{split}
\h_{\mu}[T_{A} \leq t]-\h_{\mu}[T_{B} \leq t+r\rho ] & \leq \h_{\mu}[T_{B}
\geq t+r \rho,T_{A} \leq t] \leq r^{-1} \text{ and}
\\ \h_{\mu}[T_{A} \leq \mathrm{hit}_{1-\epsilon/2,\mu}(1-q)-r\rho ] &  \leq q+r^{-1}.
\end{split}
\end{equation}
\end{lemma}
\emph{Proof.}
We first note that the first row in (\ref{eq: AB2}) follows from (\ref{eq:
AB1}) trivially. The second row in (\ref{eq: AB2}) follows from the first
by taking $t=\mathrm{hit}_{1-\epsilon/2,\mu}(1-q)-r\rho$ and picking some
$B \subseteq \Omega$ such that $\pi(B^{c}) \leq \epsilon/2$ and $\h_{\mu}[T_{B}
\leq\mathrm{hit}_{1-\epsilon/2,\mu}(1-q)] \leq q $. 

We now prove (\ref{eq: AB1}). Let $I$ be as in Lemma \ref{lem: frombigtobigger}
w.r.t.~$A$ with the choice of $k=0$. Then $\pi(I^{c}) \leq  \epsilon/2 $ and
for any $z \in I$ we have
$\mathbb{E}_{z}[T_{A} ] \leq \rho$. Denote $D:=I \cap B$. Then by the assumption
that $\pi(B^{c}) \leq \epsilon/2$, we have that $\pi(D) \geq 1-\pi(B^{c})-\pi(I^{c})
\geq 1- \epsilon$. Hence
\begin{equation}
\label{eq: ETDETA}
\mathbb{E}_{\mu}[T_{D}] \leq t_{\h,\mu}(1-\epsilon)=
\mathbb{E}_{\mu}[T_{A}]. 
\end{equation}
For any $\ell \in \R$ denote $\ell^{+}:=\max \{\ell,0 \}$. Since $D \subseteq
I$, by the Markov property we have that
$\mathbb{E}_{\mu}[(T_{A}-T_{D})^{+}] \leq \max_{z \in I} \mathbb{E}_{z}[T_A] \leq \rho$. Thus by (\ref{eq: ETDETA}) $$\mathbb{E}_{\mu}[(T_{D}-T_{A})^{+}]=\mathbb{E}_{\mu}[T_{D}-T_{A}]+\mathbb{E}_{\mu}[(T_{A}-T_{D})^{+}]\leq 0+\rho=\rho.$$
By Markov inequality and the fact that $D \subseteq B$ we get that for all $r \geq 1$
$$\h_{\mu}[T_{B}-T_{A} \geq r \rho] \leq \h_{\mu}[T_{D}-T_{A} \geq r \rho]  \leq r^{-1}.  \qed $$
\vspace{3mm}

We are now ready to prove Theorem 1.
\medskip

\emph{Proof of Theorem \ref{thm: cutoffandsetsworstinexpectation}.}
The equivalence between (i) and (ii) follows from Proposition \ref{prop: hitmix1/2intro} together with \eqref{eq: hitThetamix2}. We now show that (ii)$\Longrightarrow$(iii). Let $\alpha \in (0,1)$. Assume that $$\mathrm{hit}_{1-\alpha,\mu_{n}}^{(n)}(\epsilon)-\mathrm{hit}_{1-\alpha,\mu_n}^{(n)}(1-\epsilon)
=o \left(\mathrm{hit}_{1-\alpha,\mu}^{(n)}(1/4)
\right) , \text{ for all }0<\epsilon<1/4.$$ Then by Proposition \ref{prop: equivoftalphascutoff}, $$\mathrm{hit}_{\beta,\mu_{n}}^{(n)}(\epsilon)-\mathrm{hit}_{\beta,\mu_n}^{(n)}(1-\epsilon)
=o \left(\mathrm{hit}_{\beta,\mu_n}^{(n)}(1/4)
\right) , \text{ for all }0<\epsilon<1/4 \text{ and }0< \beta < 1.$$
Let $A_n \subseteq \Omega_n $ be an
arbitrary sequence of sets such that $\mathbb{E}_{\mu_n}[T_{A_n}]=t_{\h,\mu_n}^{(n)}(1-\alpha)
$ and $\pi_n(A_n) \geq 1-\alpha$. By the equivalence between (i) and (ii) in Theorem \ref{thm: cutoffandsetsworstinexpectation}, we have that $$\mixmun(\epsilon)-\mixmun(1-\epsilon)=o(\mixmun), \text{ for all }0<\epsilon \leq 1/4.$$ Fix some $0<\epsilon \leq 1/8$.  Using Proposition \ref{prop: equivoftalphascutoff}
and similar reasoning as in the proof of the equivalence between (i) and (ii) in Theorem \ref{thm: cutoffandsetsworstinexpectation}, we have that for all $\beta \in (0,1)$
\begin{equation}
\label{eq: worstinexpcutoff5}
(1-o(1))\mixmun\leq \mathrm{hit}_{\beta,\mu_n}^{(n)}(1-\epsilon)\leq \mathrm{hit}_{\beta,\mu_n}^{(n)}(\epsilon)\leq(1+o(1))\mixmun,
\end{equation}
Let $k_{n}(p):=\inf \{t:\h_{\mu_n}[T_{A_n} > t] \leq p \}$. Then by the
definition of $\mathrm{hit}_{1-\alpha,\mu_n}^{(n)}(\epsilon)$ and the fact
that $\pi_n(A_n) \geq 1-\alpha$ (first inequality), together with (\ref{eq:
worstinexpcutoff5}) we get that
\begin{equation}
\label{eq: furtherexplanation}
k_n(\epsilon) \leq\mathrm{hit}_{1-\alpha,\mu_n}^{(n)}(\epsilon) \leq(1+o(1))\mixmun.
\end{equation}
Conversely, let $\ell_n=\ell_n(\epsilon):=3\epsilon^{-1}(1-\alpha)^{-1}
\rel^{(n)} \log 3 =o(\mixmun)$. Then by (\ref{eq: AB2}) (first inequality)
and (\ref{eq:
worstinexpcutoff5}) we get that
\begin{equation}
\label{eq: worstinexpcutoff3}
k_n(1-2\epsilon) \geq \mathrm{hit}_{1-\alpha/2,\mu_n}^{(n)}(1-\epsilon)-\ell_n
 \geq (1-o(1))\mixmun.\end{equation}
Whence (\ref{eq: furtherexplanation}) implies that $k_n(\epsilon)-k_n(1-2\epsilon)=o(\mixmun)$. By (\ref{eq: worstinexpcutoff3}) we get that $\mixmun= O(\mathbb{E}_{\mu_n}[T_{A_n}])$ and thus we also have that $k_n(\epsilon)-k_n(1-2\epsilon)=o(\mathbb{E}_{\mu_n}[T_{A_n}])$, for all $0 < \epsilon < 1/8$. This concludes the proof of (ii)$\Longrightarrow$(iii).
We now show that  (iii)$\Longrightarrow$(ii).

Let $\alpha \in (0,1) $ and  $A_n \subseteq \Omega_n $ be an
arbitrary sequence of sets such that $\mathbb{E}_{\mu_n}[T_{A_n}]=t_{\h,\mu_n}^{(n)}(1-\alpha)
$ and $\pi_n(A_n) \geq 1-\alpha$. Assume that
\begin{equation}
\label{eq: Tanconcentratedaroundtmix}
\lim_{n \to \infty}\h_{\mu_{n}}[|T_{A_n}-\mixmun|<\epsilon
\mathbb{E}_{\mu_n}[T_{A_n}]]=1,
\text{ for all }\epsilon>0.
\end{equation}
As before, denote  $k_{n}(p):=\inf
\{t:\h_{\mu_n}[T_{A_n} > t]\leq p \}$. Then by (\ref{eq: Tanconcentratedaroundtmix}),
\begin{equation}
\label{eq: worstinexp6}
k_n(\epsilon )-k_{n}(1-\epsilon/2)=o(k_n(1-\epsilon)), \text{ for all }0<\epsilon<1/4.
\end{equation}
Recall that by assumption, there exists some $0<\delta<1$ such that $\rel^{(n)}=o(\mixmun(\delta))$. Fix some $0<\epsilon<\delta/4$. By (\ref{eq: thitinequality1})
we have that
\begin{equation}
\label{eq: worstinexp4}
 \mathrm{hit}_{1-\alpha,\mu_n}^{(n)}(1-\epsilon/2) \leq  \mathrm{hit}_{1-\alpha/2,\mu_n}^{(n)}(1-\epsilon)+o(\mixmun(\delta)).
\end{equation} 
As in (\ref{eq: furtherexplanation}),
\begin{equation}
\label{eq: asbefore}
k_n(1-\epsilon/2)
\leq \mathrm{hit}_{1-\alpha,\mu_n}^{(n)}(1-\epsilon/2).
\end{equation}
By (\ref{eq: hitThetamix}), we have that $$ (1-o(1)) \mixmun(\delta) \leq \mathrm{hit}_{1-\alpha/2,\mu_n}^{(n)}(\delta/2) .$$ Hence by (\ref{eq: worstinexp4})-(\ref{eq: asbefore}) we get that
\begin{equation}
\label{eq: worstinexp7}
 k_n(1-\epsilon/2)
\leq  \mathrm{hit}_{1-\alpha/2,\mu_n}^{(n)}(1-\epsilon) +o(\mathrm{hit}_{1-\alpha/2,\mu_n}^{(n)}(\delta/2)).
\end{equation}

This, in conjunction with (\ref{eq: worstinexp6}), yields that for all $0<\epsilon<\delta/4$,
\begin{equation}
\label{eq: worstinexp9}
k_n(\epsilon )-k_{n}(1-\epsilon/2)=o(\mathrm{hit}_{1-\alpha/2,\mu_n}^{(n)}(\delta/2))
\end{equation}

Conversely, let $\ell_n(\epsilon)$ be as before. Fix some $0<\epsilon<\delta/4$. Then $\ell_n(\epsilon/2)=o(\mixmun(\delta))$ and by (\ref{eq:
hitThetamix2}) $\ell_n(\epsilon/2)=o(\mathrm{hit}_{1-\alpha/2,\mu_n}^{(n)}(\delta/2))$. Similarly to the derivation
of (\ref{eq: worstinexpcutoff3}), by (\ref{eq: AB2})
\begin{equation}
\label{eq: worstinexp10}
k_n(\epsilon) \geq\mathrm{hit}_{1-\alpha/2,\mu_n}^{(n)}(\epsilon/2)-\ell_n(\epsilon/2)=\mathrm{hit}_{1-\alpha/2,\mu_n}^{(n)}(\epsilon/2)-o(\mathrm{hit}_{1-\alpha/2,\mu_n}^{(n)}(\delta/2)).
\end{equation}
This, in conjunction with (\ref{eq: worstinexp7})-(\ref{eq: worstinexp9}), implies that $$\mathrm{hit}_{1-\alpha/2,\mu_n}^{(n)}(\epsilon/2)-\mathrm{hit}_{1-\alpha/2,\mu_n}^{(n)}(1-\epsilon)=o(\mathrm{hit}_{1-\alpha/2,\mu_n}^{(n)}(\delta/2)) , \text{ for all }0<\epsilon<\delta/4. \qed $$

\section{Aldous' Example}
\label{s: examples}

We now present a version of Aldous' example (see Figure 2) for a sequence of reversible Markov chains $(\Omega_n,P_n,\pin)$ which satisfies the product condition but do not exhibit cutoff
and analyze
it. Our version of Aldous' demonstrates the behavior described in Proposition \ref{rem: worstinexp}. Namely, we show that the sequence does not exhibit cutoff although there exist $A_n \subseteq \Omega_n $ with $\pin(A_n) \geq 1/2 $ and $x_n \in \Omega_n$ satisfying $t_{\h}(1/2)=\mathbb{E}_{x_{n}}[T_{A_{n}}] $ such that the hitting times of $A_n$ started from $x_n$ are concentrated under the initial starting positions $x_{n} \in \Omega_n$.
\begin{figure*}[h]
\begin{center}
\includegraphics[height=5cm,width=7.5cm]{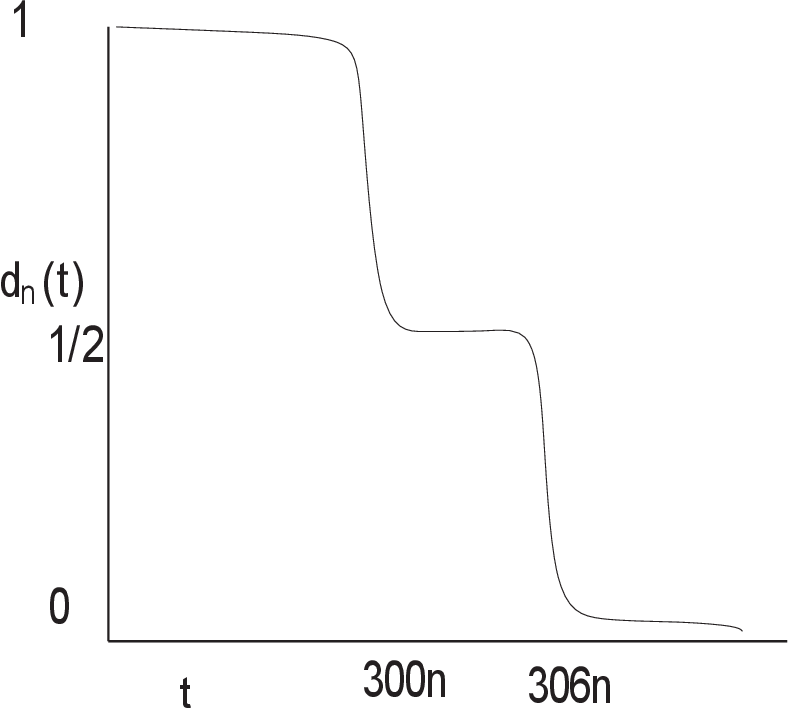}
\caption{Decay in total variation distance for Aldous' example: it does not
have cutoff.}
\label{f:dtjump}
\end{center}
\end{figure*}
\begin{figure*}[h]
\begin{center}
\includegraphics[height=9cm,width=13.1cm]{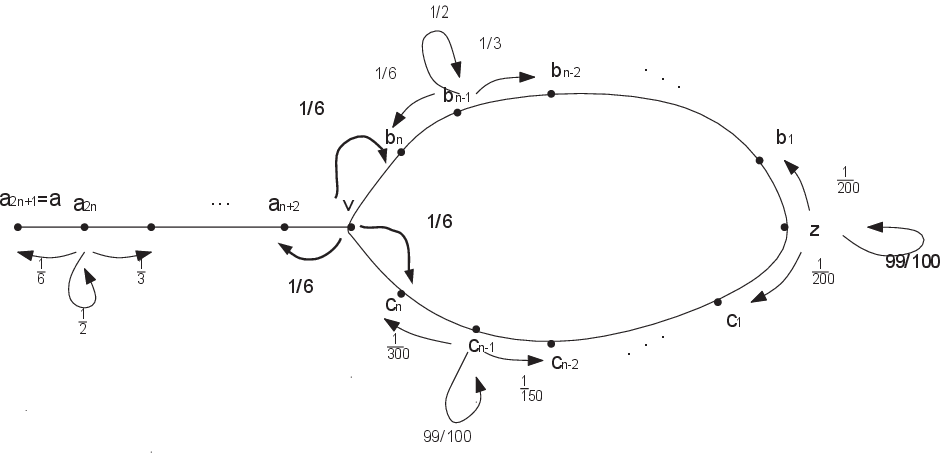}
\caption{We consider a Markov chain with the
transition probabilities specified above.}
\label{f:Aldousexample}
\end{center}
\end{figure*}
\begin{example}
\label{ex: aldous}
Consider the sequence of chains $(\Omega_n,P_n,\pi_n)$, where  $\Omega_{n}:=A
\cup B \cup C \cup \{z\} $, where $A=A_n:=\{a_{2n+1},a_{2n},a_{2n-1},\ldots,a_{n+1} \} $, $B=B_n:=\{b_n,b_{n-1},\ldots,b_1
\} $ and $C=C_n:=\{c_n,c_{n-1},\ldots,c_1
\} $. For notational convenience we write $a:=a_{2n+1}$, $v:=a_{n+1}= b_{n+1}=c_{n+1}$ and $b_0=z=c_0 $.

Define the transition matrix $P_n$ by

\begin{itemize}
\item
Holding probabilities: $$P_{n}(x,x)=\begin{cases}1/2 & x \in A_n \cup B_n, \\
99/100 & x \in C_n \cup \{z\}.
\end{cases}$$
\item
Values at the special three states $a=a_{2n+1},v=a_{n+1},z=b_{0}=c_0$:
 $P_{n}(a,a_{2n})=1/2$, $$P_n(v,a_{n+2})=P_n(v,b_{n})=P_n(v,c_{n})=\frac{1}{6},$$
 $$P_n(z,b_{1})=\frac{1}{200} = P_n(z,c_{1})=\frac{1}{200}.$$

\item
 $P_n(a_{n+k},a_{n+k-1})=P_{n}(b_{k},b_{k-1})=1/3=2P_n(a_{n+k},a_{n+k+1})=2P_{n}(b_{k},b_{k+1})$ for all $k \in [n]$.
\item
 $P_n(c_{k},c_{k-1})=\frac{1}{150} =2P_n(c_{k},c_{k+1})$
for all $k \in [n]$.
\end{itemize}
Think of $\Omega_{n}$ as a nearest neighbor walk (on an interval of length $2n+1$), biased towards state $z$, which in the middle of the interval (at state $v$) splits into
 two parallel paths, $B$ and $C$ (which we refer to as ``branches") of length
$n$ leading to $z$ (see Figure \ref{f:Aldousexample}). The difference between the two branches is that on branch $C$ the holding probability is much larger (i.e.~$P_n(c,c)=99/100 $ for all $c \in C$, while $P_n(b,b)=1/2$ for all $b \in B$).

Conditionally on not making a lazy step, the chain moves with a fixed bias
towards $z$. More precisely, let
$(\Omega_n,Q_n,\pin')$ be the non-lazy
version of  $(\Omega_n,P_n,\pin)$. That is, $Q_n(x,x)=0 $ for all $x$ and
$Q_n(x,y)=\frac{P_n(x,y)}{1-P_n(x,x)}$ for all $x \neq y$. Let $f:\Omega_n \to \{0,1,\ldots,2n+1 \}$ be $f(b_i)=i=f(c_i)$ and $f(a_{n+1+i})=n+1+i $ for all $0
\leq i \leq n$. Let $(Y_t)$ be a realization
of $(\Omega_n,Q_n,\pin')$. It is easy to see that the projection
$Z_t=f(Y_{t}) $ is a nearest neighbor biased random walk on the interval $\{0,1,\ldots,2n+1 \} $ (with reflecting boundary conditions) with a fixed bias of $2/3$ of making a step towards $0$. In particular, $T_{0}$ under $\h_{k}$ (w.r.t.~the chain $(Z_t)$) is concentrated around $\mathbb{E}_{k}[T_0]=3k - O(1) $ within a time window of size $O(\sqrt{n})$ (where the equality $\mathbb{E}_{k}[T_0]=3k
- O(1) $ follows from (\ref{eq: bdformula}) below).

It is easy to check that the chains $(\Omega_n,P_n,\pi_n)$ are indeed reversible. One way to see this is to note that Kolmogorov's cycle condition holds. Alternatively, the corresponding (symmetric) edge weights are $w_n(a_{n+m},a_{n+m+1})=2^{-(n+m)} $, $w_n(b_m,b_{m+1})=2^{-m}=w_n(c_m,c_{m+1}) $ and $w_n(x,x)=\begin{cases}\sum_{y:y \neq x}w_n(x,y) & x \in A \cup B,  \\
99 \sum_{y:y
\neq x}w_n(x,y) & \text{otherwise}. \\
\end{cases}$ 

By the well-known discrete analog of Cheeger inequality (e.g.~\cite{levin2009markov} Theorem 13.14), $t_\mathrm{rel}^{(n)}=O(1)$,
as the bottleneck-ratio is bounded from below (which can readily be seen from the above edge weights). In particular, the product
condition holds. 

For  $0<\epsilon <1$ let $k_{n}(\epsilon):=\inf \{t: \max_{x \in \Omega_n}\Pr_x[T_{z}>t] \leq \epsilon \}$.  As $\pi_n(z)>1/2$, we get that for all $0<\epsilon<1$ we have  $\mathrm{hit}_{1/2}^{(n)}(\epsilon)=k_n(\epsilon)$. To see this, observe that if $A $ is such that $\pi(A) \geq 1/2 $ then it must be the case that $z \in A $ and hence for all $x \in \Omega$ and $t \geq 0$ we have that $\max_{A \subseteq \Omega:\pi(A) \geq 1/2 }\Pr_x[T_A>t] \leq \Pr_x[T_z>t]  $. Conversely, $\pi( \{z\}) \geq 1/2 $ and so the opposite inequality holds as well.  

 We define $\mathrm{CB}$ (a shorthand for ``chosen branch") to equal $B$ (resp.~$C$) if the first visit to $z$ was made by crossing the edge $(b_{1},z)$ (resp.~$(c_{1},z)$). 

Note that for all $x \in A$ we have that $\h_{x}[\mathrm{CB}=B]=1/2=\h_{x}[\mathrm{CB}=C] $. Let $S \in \{B,C \}$. It is easy to see for every $\ell \in [n] $, conditioned on $\mathrm{CB}=S$, the conditional distribution of $T_z$ under $\h_{a_{n+1+\ell}}[\, \cdot \mid  \mathrm{CB}=S] $, is concentrated around $6 \ell +6n1_{S=B}+300n1_{S=C}$. 

Using the aforementioned projection $(Z_t)$ together with elementary results about hitting probabilities for a nearest neighbor biased walk on a segment (e.g.~\cite{levin2009markov} Exam.~9.9) we get that
\begin{equation}
\label{eq: CB1}
\h_{b_{\ell}}[T_v < T_{z}]=\frac{2^{\ell}-1}{2^{n+1}-1} =\h_{c_{\ell}}[T_v < T_{z}], \text{ for all }\ell \in [n].
\end{equation}
Consequently,
\begin{equation}
\label{eq: CB2}
\h_{c_{\ell}}[\mathrm{CB}=B]=\frac{1}{2}\cdot \frac{2^{\ell}-1}{2^{n+1}-1} =\h_{b_{\ell}}[\mathrm{CB}=C], \text{ for all }\ell \in [n].  \end{equation}
In particular, we get that for all $\ell > \lceil \log_2 n \rceil$ the law of $T_z$ under $\h_{c_{n+1-\ell}}$ (resp.~$\h_{b_{n+1-\ell}}$) is concentrated around $300(n-\ell)$ (resp.~$6(n-\ell)$), within a time window of size $O(\sqrt{n})$ .

Let $S \in \{B,C \}$. It follows from (\ref{eq: EcTz5}) below that $\mathbb{E}_{c_{n+1-r}}[T_{v}
\mid \mathrm{CB}=B ] \leq 300r+O(1)$, for all $0 \leq r \leq n$. Using Markov inequality, and the analysis of the case $x \in A$, with $x=a_{n+1}=v$,
it is easy to verify that for all  $0 \leq \ell \leq
\lceil \log_2 n \rceil$, conditioned on $\mathrm{CB}=S $, the conditional distribution of $T_z$
under $\h_{c_{n+1-\ell}}[\, \cdot \mid  \mathrm{CB}=S] $ is concentrated around $6n 1_{S=B}+300n1_{S=C}$. The same holds for $b_{n+1-\ell}$.

By the analysis above, 
\begin{equation}
\label{eq: k34}
k_{n}(1/2-o(1)) \geq 306n-o(n)  \text{ and } k_{n}(1/2+o(1)) \leq 300n+o(n). 
\end{equation} 
In particular, there is no $\hit_{1/2}$-cutoff. By Theorem \ref{thm: psigmacutoffequiv}, the sequence does not exhibit a cutoff.

\medskip 

 Let $x_n \in \Omega_n$ be such that ~$t_{\h}^{(n)}(1/2)=\max_{y \in \Omega_n} \mathbb{E}_{y}[T_{z}]=\mathbb{E}_{x_{n}}[T_{z}]$. We now argue that the $x_{n}=c_{j_{n}}  $ for some $j_n \in [n] $ such that $\min (n-j_{n},j_n ) \to \infty $ (in fact, we shall show that $n-j_{n} = \Theta (\log n)$). Note that starting from such $x_n$, the hitting time of $z$ is concentrated,
although the sequence
of chain does not exhibit a cutoff.

Most readers should be satisfied by the following explanation. It is clear that either $x_n \in C $ or $x_n=a$. If $\ell_{n} = o(n) $ and $\ell_n \to \infty$, then the distribution of $T_z$ under $\h_{c_{n-\ell_n}}$ is concentrated around $300n-o(n)$ and $\mathbb{E}_{c_{n-\ell_n}}[T_{z}]=300n-o(n) $. On the other hand, $\mathbb{E}_{a}[T_{z}] \leq 159n$. Lastly, if $\ell_n=O(1) $, then $\h_{c_{n+1-\ell_{n}}}[  \mathrm{CB}=B] $ is bounded from below, and so $\limsup_{n \to \infty} \mathbb{E}_{c_{n+1-\ell_n}}[T_{z}]/n < 300 $.

We now present a more detailed proof for the fact that $x_n=c_{j_{n}}
 $ for some $j_n  $ such that $n-j_{n} = \Theta (\log n)$.  First write
\begin{equation}
\label{eq: EcTz1}
\mathbb{E}_{c_{ r}}[T_{z}]=\mathbb{E}_{c_{
r}}[T_{z} \mid T_z < T_v ]\frac{2^{n+1}-2^{r}}{2^{n+1}-1}+\mathbb{(E}_{c_{r}}[T_{v} \mid T_v < T_z ]+\mathbb{E}_{v}[T_{z} ])\frac{2^{r}-1}{2^{n+1}-1}.
\end{equation}
We shall show
that there exist absolute constants $K_1,K_{2},K_{3}>0 $ such
that for all $r \in [n] $ \begin{equation}
\label{eq: bdformula2}
300-K_1 2^{-r} \leq \mathbb{E}_{c_{r+1}}[T_{c_{
r}} ] \leq 300 \text{ and } 6-K_1 2^{-r} \leq \mathbb{E}_{a_{n+r+1}}[T_{a_{
n+r}} ], \mathbb{E}_{b_{r+1}}[T_{b_{
r}} ] \leq 6.
\end{equation}
\begin{equation}
\label{eq: bdformula3}
\frac{300 n}{2}+\frac{6 n}{2} -K_2 \leq \mathbb{E}_{v}[T_{z} ] \leq \frac{300 (n+1)}{2}+\frac{6 (n+1)}{2} =153(n+1).
\end{equation}
\begin{equation}
\label{eq: bdformula3'}
\mathbb{E}_{a}[T_{z} ]=\mathbb{E}_{a}[T_{v} ]+  \mathbb{E}_{v}[T_{z} ] \leq 159(n+1).
\end{equation}
\begin{equation}
\label{eq: EcTz2}
\mathbb{E}_{c_{n+1-r}}[T_{c_{n-
r}} \mid T_z < T_v ]=300\pm K_{3} 2^{-r}=\mathbb{E}_{c_{r}}[T_{c_{
r+1}} \mid T_v < T_z ].
\end{equation}
\begin{equation}
\label{eq: EcTz4}
\mathbb{E}_{c_{n+1-r}}[T_{z} \mid T_z < T_v ]=300 (n+1-r) \pm 2 K_{3} 2^{-r}.
\end{equation}
\begin{equation}
\label{eq: EcTz5}
\mathbb{E}_{c_{n+1-r}}[T_{v} \mid T_v < T_z ]=300r \pm2 K_{3} 2^{-(n-r)}.
\end{equation}  
Combining (\ref{eq: bdformula3})-(\ref{eq: EcTz5}) with (\ref{eq: EcTz1}) it is easy to verify that indeed  $x_{n}=c_{j_{n}}
 $ for some $j_n \in [n] $ such that $n-j_{n} = \Theta (\log n)$. 

We first note that (\ref{eq: bdformula3}) and (\ref{eq: bdformula3'}) follow easily from (\ref{eq: bdformula2}). We now prove (\ref{eq: bdformula2}). 
 
It is a standard result (e.g.~\cite{aldous2000reversible} Lemma 1 Chapter
5, or \cite{cutoff} Lemma 5.2) that for a birth and death chain on $\{0,1,\ldots,2n+1 \} $
with symmetric edge weights $(w_{i,j})_{i,j : |i-j| \leq 1}$, 
\begin{equation}
\label{eq: bdformula}
\mathbb{E}_{r+1}[T_{r}]= \frac{\sum_{i,j:i \geq r,j \geq r+1,|i-j| \leq 1
}w_{i,j}}{w_{r,r+1}}-1.
\end{equation}
It follows from (\ref{eq: bdformula}) that the projected chain $(Z_t) $ satisfies that $3-2^{-(2n-k)} \leq \mathbb{E}_{k+1}[T_{k}] \leq 3 $, for all $0 \leq k \leq 2n$ which together with (\ref{eq: CB1}) imply (\ref{eq: bdformula2}). As (\ref{eq: EcTz4}) and (\ref{eq: EcTz5}) follow from (\ref{eq: EcTz2}), we conclude the proof by verifying (\ref{eq: EcTz2}). 

Using the Doob's transform (see e.g.~\cite{levin2009markov} Section 17.6) we have that the law of $(X_{t })$ up to time $T_v$ (resp.~$T_z$) conditioned on $T_v < T_z $ (resp.~$T_z<T_v $) is a Markov chain whose transition matrix is given by $P_v(x,y):=\frac{P_{n}(x,y)\h_{y}[T_v<T_z] }{\h_{x}[T_v<T_z]} $ (resp.~$P_z(x,y)=\frac{P_{n}(x,y)\h_{y}[T_z<T_v] }{\h_{x}[T_z<T_v]} $).

By (\ref{eq: CB1}) we get that for all $ r \in [  n] $ there exists an absolute constant $K_{4}>0$ such that  $P_z(c_{n+1-r},y)=P_n(c_{n+1-r},y) \pm K_{4}2^{- r} $ for all $y $, while $P_v(c_{r},c_{r+i})=P_n(c_{r},c_{r-i})\pm K_{4}2^{- r}$ for $i \in \{0,\pm 1\}$ (i.e., up to negligible
terms, $P_v$ restricted to $C$ is a nearest neighbor walk with an opposite bias compared to the original chain). This, in conjunction with (\ref{eq: bdformula2}), implies (\ref{eq: EcTz2}).
\end{example}
\section*{Acknowledgements}
The author would like to thank David Aldous, Riddhipratim Basu, Yuval Peres and Allan Sly for many useful discussions. We are also grateful for the anonymous referee for suggesting some improvements to the presentation.

\bibliographystyle{plain}
\bibliography{tave}

\begin{thebibliography}{10}

\bibitem{aldous1982some}
David Aldous.
\newblock Some inequalities for reversible {M}arkov chains.
\newblock {\em J. London Math. Soc. (2)}, 25(3):564--576, 1982.
\newblock \href{http://www.ams.org/mathscinet-getitem?mr=MR657512}{MR657512}.

\bibitem{aldous2000reversible}
David Aldous and Jim Fill.
\newblock Reversible {M}arkov chains and random walks on graphs, 2002.
\newblock Unfinished manuscript.

\bibitem{cutoff}
Riddhipratim Basu, Jonathan Hermon, and Yuval Peres.
\newblock Characterization of cutoff for reversible {M}arkov chains.
\newblock {\em Ann. Probab.}, 45(3):1448--1487, 2017.
\newblock \href{http://www.ams.org/mathscinet-getitem?mr=MR3650406}{MR3650406}.

\bibitem{chenphd}
Guan-Yu Chen.
\newblock {\em The cutoff phenomenon for finite {M}arkov chains}.
\newblock ProQuest LLC, Ann Arbor, MI, 2006.
\newblock Thesis (Ph.D.)--Cornell University.
  \href{http://www.ams.org/mathscinet-getitem?mr=MR2708835}{MR2708835}.

\bibitem{chen2013comparison}
Guan-Yu Chen and Laurent Saloff-Coste.
\newblock Comparison of cutoffs between lazy walks and {M}arkovian semigroups.
\newblock {\em J. Appl. Probab.}, 50(4):943--959, 2013.
\newblock \href{http://www.ams.org/mathscinet-getitem?mr=MR3161366}{MR3161366}.

\bibitem{spectral}
Sharad Goel, Ravi Montenegro, and Prasad Tetali.
\newblock Mixing time bounds via the spectral profile.
\newblock {\em Electron. J. Probab.}, 11:no. 1, 1--26, 2006.
\newblock \href{http://www.ams.org/mathscinet-getitem?mr=MR2199053}{MR2199053}.

\bibitem{griffiths2012tight}
Simon Griffiths, Ross~J. Kang, Roberto~Imbuzeiro Oliveira, and Viresh Patel.
\newblock Tight inequalities among set hitting times in {M}arkov chains.
\newblock {\em Proc. Amer. Math. Soc.}, 142(9):3285--3298, 2014.
\newblock \href{http://www.ams.org/mathscinet-getitem?mr=MR3223383}{MR3223383}.

\bibitem{levin2009markov}
David~A. Levin and Yuval Peres.
\newblock {\em Markov chains and mixing times}.
\newblock American Mathematical Society, Providence, RI, 2017.
\newblock Second edition of [ MR2466937], With contributions by Elizabeth L.
  Wilmer and a chapter on ``Coupling from the past'' by James G. Propp and
  David B. Wilson.
  \href{http://www.ams.org/mathscinet-getitem?mr=MR3726904}{MR3726904}.

\bibitem{lovasz1998mixing}
L\'aszl\'o Lov\'asz and Peter Winkler.
\newblock Mixing times.
\newblock In {\em Microsurveys in discrete probability ({P}rinceton, {NJ},
  1997)}, volume~41 of {\em DIMACS Ser. Discrete Math. Theoret. Comput. Sci.},
  pages 85--133. Amer. Math. Soc., Providence, RI, 1998.
\newblock \href{http://www.ams.org/mathscinet-getitem?mr=MR1630411}{MR1630411}.

\bibitem{tetali}
Ravi Montenegro and Prasad Tetali.
\newblock Mathematical aspects of mixing times in {M}arkov chains.
\newblock {\em Found. Trends Theor. Comput. Sci.}, 1(3):237--354, 2006.
\newblock \href{http://www.ams.org/mathscinet-getitem?mr=MR2341319}{MR2341319}.

\bibitem{oliveira2012mixing}
Roberto~Imbuzeiro Oliveira.
\newblock Mixing and hitting times for finite {M}arkov chains.
\newblock {\em Electron. J. Probab.}, 17:no. 70, 12, 2012.
\newblock \href{http://www.ams.org/mathscinet-getitem?mr=MR2968677}{MR2968677}.

\bibitem{peres2011mixing}
Yuval Peres and Perla Sousi.
\newblock Mixing times are hitting times of large sets.
\newblock {\em J. Theoret. Probab.}, 28(2):488--519, 2015.
\newblock \href{http://www.ams.org/mathscinet-getitem?mr=MR3370663}{MR3370663}.

\bibitem{starr1966operator}
Norton Starr.
\newblock Operator limit theorems.
\newblock {\em Trans. Amer. Math. Soc.}, 121:90--115, 1966.
\newblock \href{http://www.ams.org/mathscinet-getitem?mr=MR0190757}{MR0190757}.

\end{thebibliography}
\end{document}